\documentclass[11pt]{article}
\usepackage{graphicx}%
\usepackage{hyperref}%
\usepackage{comment}%
\usepackage{multirow}%
\usepackage{amsmath,amssymb,amsfonts}%
\usepackage{amsthm}%
\usepackage{mathrsfs}%
\usepackage[title]{appendix}%
\usepackage{xcolor}%
\usepackage{textcomp}%
\usepackage{manyfoot}%
\usepackage{booktabs}%
\usepackage{subcaption}
\usepackage{listings}%
\usepackage{algorithm}
\usepackage{algorithmic}
\newtheorem{theorem}{Theorem}
\newtheorem{corollary}{Corollary}

\newtheorem{remark}{Remark}

\newtheorem{assumption}{Assumption}
\newtheorem{definition}{Definition}

\begin{document}

\title{Riemannian Inexact Gradient Descent for Quadratic Discrimination}

\author{
  Uday Talwar \\
  School of Mathematical Sciences, University of Arizona \\
  \texttt{udaytalwar@arizona.edu}
  \and
  Meredith K. Kupinski \\
  Wyant College of Optical Sciences, University of Arizona \\
  \texttt{meredithkupinski@arizona.edu}
  \and
  Afrooz Jalilzadeh \\
  Department of Systems and Industrial Engineering, \\ University of Arizona \\
  \texttt{afrooz@arizona.edu}
}

\maketitle
\begin{abstract}
    We propose an inexact optimization algorithm on Riemannian manifolds, motivated by quadratic discrimination tasks in high-dimensional, low-sample-size (HDLSS) imaging settings. In such applications, gradient evaluations are often biased due to limited sample sizes. To address this, we introduce a novel Riemannian optimization algorithm that is robust to inexact gradient information and prove an $\mathcal O(1/K)$ convergence rate under standard assumptions. We also present a line search variant that requires access to function values but not exact gradients, maintaining the same convergence rate and ensuring sufficient descent. The algorithm is tailored to the Grassmann manifold by leveraging its geometric structure, and its convergence rate is validated numerically. A simulation of heteroscedastic images shows that when bias is introduced into the problem, both intentionally and through estimation of the covariance matrix, the detection performance of the algorithm solution is comparable to when true gradients are used in the optimization. The optimal subspace learned via the algorithm encodes interpretable patterns and shows qualitative similarity to known optimal solutions. By ensuring robust convergence and interpretability, our algorithm offers a tool for manifold-based dimensionality reduction in the presence of inexact gradients due to sample statistics.
\end{abstract}

\section{Introduction}\label{sec1}

Optimal binary classification requires the likelihood ratio of class-conditional probability density functions as a test statistic. Modern imaging systems produce high-dimensional data—often millions of elements per image—making it computationally expensive or infeasible to estimate optimal test statistics, especially when training data is limited. A practical alternative is to use test statistics that are linear or quadratic in the data, both of which rely on estimating the mean and covariance for each class from the training set. However, in high-dimensional, low-sample-size (HDLSS) settings, these estimates are often unreliable due to insufficient training samples.


One strategy to address this challenge is dimensionality reduction via a linear transformation. This raises the question: which transformation matrix best projects the data into a lower-dimensional space while preserving class-discriminative information? We consider a linear transformation of the form
\begin{align*}
    v = Tg,
\end{align*}
where $T$ is an $p \times n$ ($p << n$) matrix representing a point on a Grassmann manifold ($\mathrm{Gr}(p,n)$)
- a particular type of Riemannian manifold, $g$ is an $n \times 1$ image vector and $v$ the channelized representation of the image. The Grassmann manifold, denoted by $\mathrm{Gr}(p,n)$ in general, is the set of all $p$-dimensional linear subspaces of $\mathbb{R}^n$, and has various applications in machine learning, image processing, low-rank matrix approximation, and model reduction. Any figure of merit must satisfy the invariance condition \cite{kupinski_15}
\begin{align}\label{inherent_grass}
    f(MT) = f(T),
\end{align}
for any full-rank matrix M, implying that $f$ depends only on the subspace spanned by the rows of 
T. Thus, the optimization problem is inherently defined over the Grassmann manifold. While our motivating example lies on the Grassmannian, we consider a more general Riemannian manifold optimization framework, which can later be specialized to this setting. To that end, we formulate the problem as minimizing a smooth (possibly nonconvex) function 
$f$ over a general Riemannian manifold $\mathcal M$:
\begin{align}\label{main}\min_{X \in \mathcal{M}} f(X),\end{align}
where $\mathcal{M}$ is a Riemannian manifold and $f$ is geodesically Lipschitz-smooth. 

 Gradient-based methods are commonly used for manifold optimization, but in HDLSS settings, gradient estimates can be biased due to limited data. This undermines the effectiveness of conventional stochastic methods, which rely on unbiased gradients. To address this, our work introduces the Riemannian Inexact Gradient Descent (RiGD) algorithm, offering a robust alternative when exact or unbiased gradients are unavailable.

After introducing our algorithm for Riemannian manifolds, it is further specialized for the Grassmann where closed-form expressions of relevant operators are available. In the numerical section, we explore the application of our algorithm to optimal linear image compression with the goal of quadratic test statistics for detection. 

\subsection{Related work}\label{related work}

The term HDLSS was coined by Hall et al. \cite{hall_hdlss}, where the authors noted the emergence of data sets with large (and growing) dimension coupled with low (or fixed) sample sizes. It is well understood that the task of estimation becomes exceedingly difficult in such settings, motivating the need to reduce the dimension of the problem to make statistical analysis tractable. With technological advancements in imaging, images today have large dimensions often coupled with a small number of samples - examples include medical imaging and text processing (see \cite{HDLSS_1}, \cite{HDLSS_2}, \cite{HDLSS_3}). To tackle the HDLSS image setting, one may represent an image by a linear subspace which lies in the Grassmann manifold or Grassmannian. An extension of linear discriminant analysis to problems in which data consist of linear subspaces was presented in \cite{hamm2008grassmann}, called Grassmann Discriminant Analysis (GDA). GDA embeds the Grassmann into a higher dimensional Hilbert space and then uses algorithms suited for Hilbert spaces.

In \cite{huang_2015}, the authors propose a projection-based approach, transforming data in a higher dimensional Grassmann manifold to a lower, more discriminant Grassmann manifold by learning a transformation matrix, making the approach different from other traditional methods which rely on embedding the Grassmannian in a higher dimension Hilbert space. As for the optimization algorithm, the authors use the established nonlinear Riemannian Conjugate Gradient (RCG) method (we refer the reader to \cite{edelman1998}, \cite{AbsMahSep2008}, \cite{boumal2023intromanifolds} for a rigorous treatment of Riemannian algorithms). This work is expanded in \cite{wang_2017} where the authors propose Structure Maintaining Discriminant Maps (SMDM), extending the notion of dimension reduction from Euclidean space to manifolds. Again, the optimization problem is solved using the RCG algorithm.

In this paper, our numerical work considers the problem presented in Kupinski et al. \cite{kupinski_15} where the authors use the symmetrized Kullback-Leibler (KL) or Jeffrey's divergence for binary classification. Advantages of the proposed method include closed form expressions for the objective function and gradient as well as computational efficiency via dimension reduction. We note that \cite{kupinski_15} did not present novel optimization algorithms and in our work we utilize the Riemannian geometry of the Grassmann manifold through the manifold optimization approach. The optimal matrix $T^*$ facilitates dimension reduction while maximizing Jeffrey's divergence between the two probability distributions in the lower-dimensional subspace. The proposed algorithm is designed for HDLSS settings where errors in gradient calculation impact algorithm performance.

Early work to adapt Euclidean algorithms to Riemannian manifolds may be traced to \cite{Luenberger1972}, \cite{Luenberger1973-kb}. Gabay \cite{Gabay1982-lv} introduced steepest descent along geodesics, Newton's method and Quasi-Newton methods for Riemannian manifolds. Udriste \cite{Udrişte1994} introduced a general descent algorithm on Riemannian manifolds leveraging the exponential map. Bonnabel \cite{Bonnabel_2013} extends the stochastic gradient descent approach to Riemannian manifolds. The convergence rate for SGD on Riemannian manifolds was improved in \cite{tripuraneni18a}, utilizing an averaging approach. Further improvements to the rate were provided in \cite{zhang2017riemanniansvrgfaststochastic} where the authors present Stochastic Variance Reduced Gradient for Riemannian manifolds. Projection-free extensions with additional constraints (not only the parameter domain being a manifold) for nonconvex and geodesically convex problems were presented by Weber \cite{weber2019nonconvex} along with convergence analysis. We refer the reader to \cite{Hosseini2020} for a broad overview of stochastic optimization of Riemannian manifolds. Given the inherent noise in our setting, one may consider leveraging stochastic gradient methods. However, stochastic methods rely on unbiased estimators of the objective function and we will show that is not the case for our setting. To our knowledge, there are no known algorithms that address biased estimates of gradients for optimization over Riemannian manifolds. Thus, we introduce the Riemannian Inexact Gradient Descent algorithm, which utilizes an inexact approach to gradient-based optimization. While preparing the final version of this manuscript, we became aware of an independent and concurrent work by Zhou et al. \cite{zhou2024inexactriemanniangradientdescent}, which also investigates inexact gradient methods for manifold optimization. Their analysis establishes convergence and, under normalized conditions and the Kurdyka–Łojasiewicz (KL) property, derives a convergence rate. In contrast, our work proves convergence rates without requiring such assumptions, thereby broadening the applicability of the theoretical guarantees. Moreover, our method incorporates an adaptive line search strategy that improves practical performance and eliminates the need for stepsize tuning. Finally, our approach is motivated by and tailored to applications in imaging, offering a complementary perspective to the theoretical developments. 

For the Euclidean setting, \cite{schmidt2011convergenceratesinexactproximalgradient} introduces an inexact proximal algorithm where an error is present in the calculation of the gradient of the smooth term or in the proximity operator with respect to the non-smooth term. Convergence analysis in \cite{schmidt2011convergenceratesinexactproximalgradient} shows that the inexact method achieves the same convergence rate as in the error-free case. In our work, we consider errors in the gradient while minimizing a function over a Riemannian manifold with the constraint set determined by the geometry of the parameter domain. 

\subsection{Contribution}\label{contribution}

Motivated by the absence of an inexact gradient-based optimization approach for Riemannian manifolds, we introduce the \textit{Riemannian Inexact Gradient Descent} algorithm for optimizing functions defined on Riemannian manifolds when only approximate gradient information is available. Our contributions are summarized as follows:
\begin{itemize}
    \item \textbf{Convergence Analysis:}  Assuming standard smoothness and the existence of a minimizer, we establish an $\mathcal{O}(1/K)$ convergence rate for both constant and diminishing step sizes, provided the gradient approximation error diminishes at an appropriate rate. This rate matches that of Riemannian gradient descent when true gradients are known (\cite{AbsMahSep2008} for details).
    
    \item \textbf{Bounded Iterates:} Under mild and commonly used assumptions in Riemannian optimization, we show that the sequence of iterates generated by RiGD remains bounded while preserving the $\mathcal{O}(1/K)$ convergence rate.

    \item \textbf{Line Search Variant:} When only inexact gradient information is available but objective values can be accessed, we introduce a line search method (RiGD-LS) that achieves an $\mathcal{O}(1/K)$ rate and satisfies a sufficient decrease condition.
    
    \item \textbf{Application to Imaging:} We demonstrate the practical effectiveness of RiGD on a binary classification problem in imaging. In this setting, biased gradient information can hinder the performance of traditional gradient descent and stochastic methods, while our approach remains robust and effective.
\end{itemize}

In the next section, mathematical preliminaries and requisite assumptions are established along with key definitions. Then, we present our optimization algorithm without line search and establish its convergence guarantees in Section~\ref{proposed method}. The line search variant and its corresponding analysis are discussed in Section~\ref{sec:line search}. In Section~\ref{sec:grassmann}, we specialize our method to the Grassmann manifold and highlight its geometric structure. After establishing the theoretical guarantees of the algorithm, we apply it to a binary classification problem in imaging in Section~\ref{results}.

\section{Mathematical preliminaries}\label{mathematical prelims}
 
\begin{figure}[!h]
\centering
\begin{minipage}{0.65\textwidth}
Denote by $\mathcal{M}$ a smooth differentiable manifold of dimension $n$ that locally resembles Euclidean space i.e. every $x \in \mathcal{M}$ has a neighborhood $U \subset \mathcal{M}$ that is homeomorphic to an open subset of $\mathbb{R}^n$. For any point $x \in \mathcal{M}$, the associated tangent space is denoted by $\mathcal{T}_x\mathcal{M}$ which contains the set of tangent vectors at $x$. A geodesic $\gamma$ is a locally length-minimizing curve; on a plane, geodesics are straight lines, and on a sphere, they are segments of great circles. In Figure \ref{manifold1}, a sphere is presented along with its associated tangent space at a given point and a geodesic between two points on the manifold. The blue curve represents a geodesic $\gamma$ between points $x$ and $y$ on the manifold. In red, we show a tangent vector $g$ that belongs to the tangent space.

A fundamental operation in manifold optimization is the retraction map. One important example is the Exponential Map, defined as follows.
\end{minipage}
\hfill
\begin{minipage}{0.3\textwidth}
\centering
\fbox{\includegraphics[width=\linewidth]{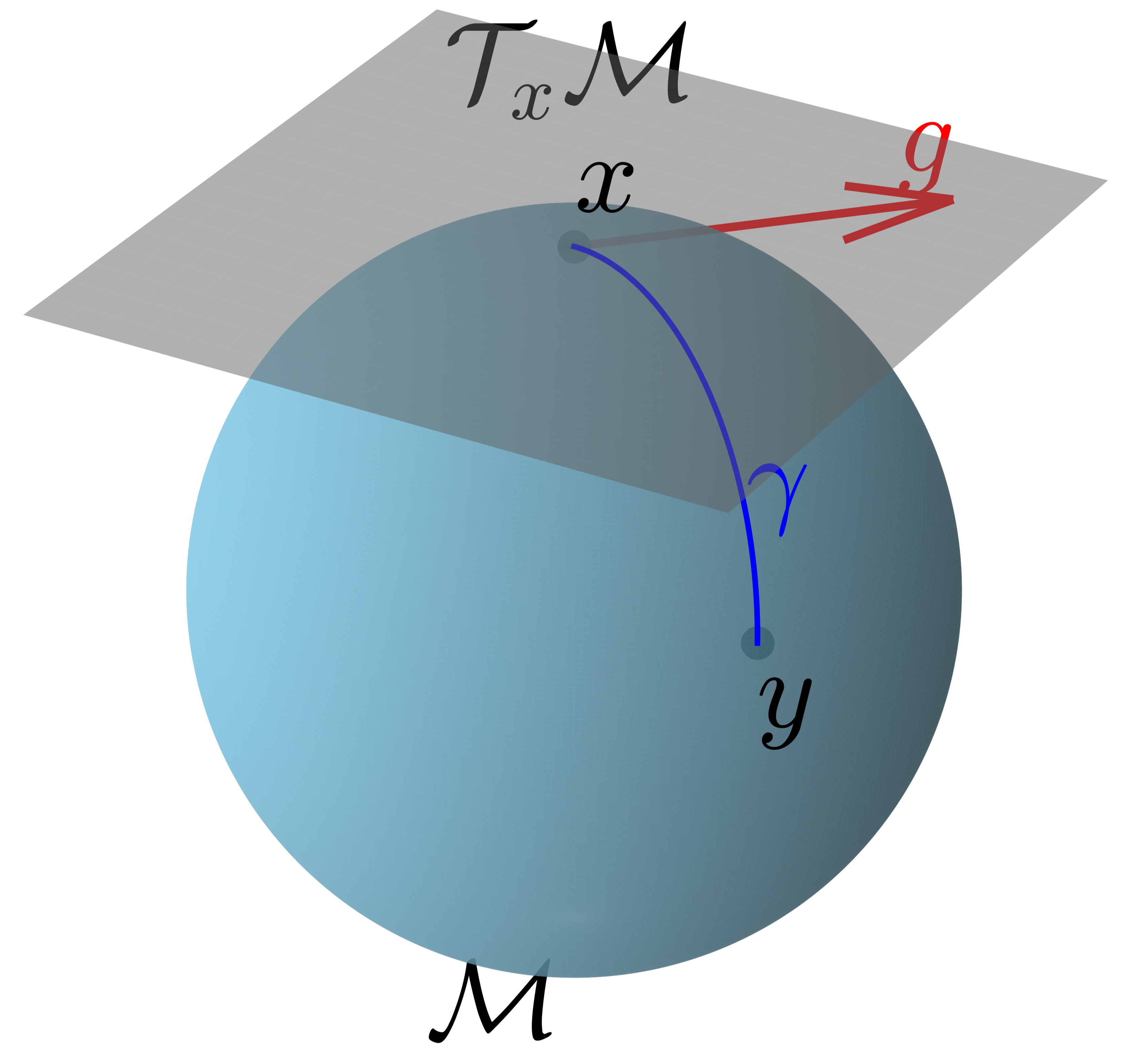}}
\captionof{figure}{A sphere manifold $\mathcal{M}$ and its associated tangent space $\mathcal{T}_x\mathcal{M}$ at point $x$. The geodesic connecting $x$ and $y$ is denoted by $\gamma$ and $g$ is a vector in $\mathcal{T}_x\mathcal{M}$. }
\label{manifold1}
\end{minipage}
\end{figure}

\begin{definition}(Exponential Map)\label{def1}
    Denote by $Exp: \mathcal{T}_x\mathcal{M} \supset U \rightarrow \mathcal{M}$ an exponential map, defined on an open neighborhood $U$ of $0 \in \mathcal{T}_x\mathcal{M}$, such that for any $v \in U$, $y = Exp_x(v) \in \mathcal{M}$ lies on the geodesic $\gamma: [0,1] \mapsto \mathcal{M}$ satisfying $\gamma(0) = x$, $\gamma(1) = y$, and $\gamma'(0) = v$.
\end{definition} 
Additionally, the {\it{inverse exponential map}} is defined as $Exp^{-1}_x: \mathcal{M} \supset V \rightarrow \mathcal{T}_x\mathcal{M}$, where $V$ is a neighborhood of $x \in \mathcal{M}$ such that the exponential map and its inverse are continuously differentiable (i.e., a local diffeomorphism), with $Exp^{-1}_x(x) = 0$.

Now let $g:\mathcal{T}_x\mathcal{M} \times \mathcal{T}_x\mathcal{M} \rightarrow \mathbb{R}$ denote the Riemannian metric, which assigns to each $x$ a positive-definite inner product $g_x(u,v) = \langle u, v \rangle_x$. Throughout the rest of our paper, we omit the subscript of the tangent space whenever it is evident from the context. The metric $g$, induces a norm $\| \cdot \|: \mathcal{T}_x\mathcal{M} \rightarrow \mathbb{R}$ defined by $\left\Vert q \right\Vert_x = \sqrt{g_x(q,q)}$ for $q \in \mathcal{T}_x\mathcal{M}$. Together, the pair $(\mathcal{M}, g)$ describe a Riemannian manifold. Another useful property is that for $y = Exp_x(v) \in \mathcal{M}$ we have $d(x,y) = \left\Vert v \right\Vert_x$, where $\gamma$ is the unique length-minimizing geodesic from $x$ to $y$ and $d$ is a distance function that satisfies positivity, symmetry and the triangle inequality. For simplicity of the notation, we use $\|\cdot\|$ instead of $\|\cdot\|_x$ throughout the paper.

The Riemannian metric also provides the natural notion of gradient on the manifold. For a differentiable function $f$ on a finite-dimensional real inner-product space, the gradient of $f$ at $x$, denoted by $\nabla f(x)$, is the unique vector such that
$Df(x)[v] = \langle \nabla f(x), v \rangle$ for all $v$.
On a Riemannian manifold $(\mathcal{M}, g)$, the gradient of $f$ at $x$, denoted by $grad(f(x))$, is the unique tangent vector in $T_x\mathcal{M}$ such that
$Df(x)[v] = g_x(grad(f(x)), v)$ for all $v \in T_x\mathcal{M}.$



\subsection{Assumptions}

We state assumptions relevant to the mathematical work presented in this paper.

\begin{assumption}\label{asu1}
(Smoothness) Let $f:\mathcal{M} \rightarrow \mathbb{R}$ be a differentiable function. Then $f$ is $L$-smooth if its gradient, $grad \ f(x)$ satisfies 
\begin{equation}
    f(y) \leq f(x) + \langle grad \ f(x), Exp_x^{-1}(y) \rangle_x + \frac{L}{2}d^2(x,y),
\end{equation}
for all $x, y \in \mathcal{M}$.
\end{assumption}

The above is a standard assumption in the optimization literature and is used in the convergence rate proof \cite{zhang2016firstordermethodsgeodesicallyconvex, weber2023riemannian}. The next assumption is required to establish the boundedness of the iterates generated by our proposed algorithm.

\begin{assumption} \label{asu3}
(Bounded Sublevel Set) Let $f:\mathcal{M} \rightarrow \mathbb{R}$ be a differentiable function. Denote by $\mathcal{L}^-_c = \{x \in \mathcal{M}: f(x) \leq c\}$ the sublevel set of $f$. The sublevel set is said to be bounded if it is contained within a geodesic ball of finite radius. More precisely, $\mathcal{L}^-_c$ is bounded if there exists $x_0 \in \mathcal{M}$ and a finite radius $R > 0$ such that
\[
\mathcal{L}^-_c \subseteq B_{R}(x_0), 
\]
where $B_{R}(x_0) = \{x \in \mathcal{M} \mid d(x, x_0) < R\}$ is a geodesic ball centered at $x_0$.
\end{assumption}
Given the boundedness of the sublevel set in Assumption~\ref{asu3}, we further assume that it is closed to ensure the existence of a minimizer.
\begin{assumption} \label{asu2}
(Existence of a Minimum) Let $f:\mathcal{M} \rightarrow \mathbb{R}$ be a differentiable function. Suppose there exists $x_0 \in \mathcal{M}$ such that the sublevel set
\[
\mathcal{L}_c := \{ x \in \mathcal{M} \mid f(x) \leq f(x_0) \}
\]
is nonempty, closed, and bounded. Then, there exists $x^* \in \mathcal{L}_c$ such that
\[
f(x^*) \leq f(x), \qquad \forall x \in \mathcal{L}_c.
\]
\end{assumption}




\section{Algorithm Design and Theoretical Guarantees} \label{proposed method}

In this section, we propose the Riemannian Inexact Gradient Descent algorithm, to minimize a nonconvex function $f$, over a Riemannian manifold $(\mathcal{M}, g)$ using an inexact gradient. The details can be seen in Algorithm \ref{rigd}. 

\begin{algorithm}[htbp]
\caption{Riemannian Inexact Gradient Descent (RiGD)}\label{rigd}
\begin{algorithmic}[1]
\STATE{\bf Input:} Initial point $x_0 \in \mathcal{M}$, step size $\eta$ 
\FOR{$k = 0, \dots, K-1$}  
\STATE  $\Delta_k \leftarrow argmin_{\Delta \in \mathcal{T}_{x_k} \mathcal{M}} \left\Vert \Delta - \widetilde{grad}(f(x_k)) \right\Vert$ 
\STATE $x_{k+1} = Exp_{x_k}(-\eta_k\Delta_k)$ 
\ENDFOR
\end{algorithmic}
\end{algorithm}

At each iteration $k$, the algorithm has access to an \emph{inexact gradient}, denoted by $\widetilde{grad}(f(x_k))$. Throughout this work, we assume that the manifold admits an ambient matrix-space representation together with a well-defined projection onto the tangent space at each point. In this setting, our method allows for inexact gradient computations that may not initially belong to the tangent space. These inexact gradients are therefore projected onto the tangent space before the update is performed. Our analysis relies on this tangent-space projection mechanism, rather than on the manifold being a linear Euclidean space or globally embedded as a Euclidean submanifold. This viewpoint covers standard embedded manifolds as well as matrix-represented quotient manifolds \cite{edelman1998} such as the Grassmannian, which is the focus of our numerical section. 

To obtain a valid search direction on the manifold, we project this inexact gradient onto the tangent space, resulting in a direction $\Delta_k$. Specifically, $\Delta_k$ is the closest tangent vector to $\widetilde{grad}(f(x_k))$, i.e., 
$\Delta_k = \arg\min_{\Delta \in \mathcal{T}_{x_k} \mathcal{M}} \| \Delta - \widetilde{grad}(f(x_k)) \|$.
The next iterate, $x_{k+1}$, is obtained by moving from $x_k$ along the geodesic in the direction of $-\Delta_k$ with step size $\eta_k$, using the exponential map as $x_{k+1} = Exp_{x_k}(-\eta_k \Delta_k)$. This update ensures that the iterate remains on the manifold. The process is repeated for $K$ iterations.

\subsection{Convergence Analysis} \label{convergence analysis}
\noindent In this section we study the convergence properties of our proposed algorithm. 

{\begin{theorem}\label{th:rate}
Suppose Assumptions~\ref{asu1}-\ref{asu2} hold. Define the error term as $e_k \triangleq {grad}(f(x_k)) - \Delta_k$, and let $\{x_k\}$ be the sequence of iterates generated by Algorithm~\ref{rigd}, with step size satisfying $\eta_k \leq \frac{\alpha_k}{1 + 2\alpha_k L}$, for some $\alpha_k>0$. Then, the following holds: 
   \begin{align}\label{th:result}
    \frac{\eta_k}{2}\|grad(f(x_{k}))\|^2 \leq f(x_k)-f(x_{k+1})+\left(\frac{\alpha_k}{2}+L\eta_k^2\right)\|e_k\|^2.
    \end{align}
\end{theorem}}

\begin{proof}
Using Assumption \ref{asu1}, the fact that $d(x_k, x_{k+1}) = \left\Vert -\eta_k\Delta_k \right\Vert$ and \\ 
$Exp_{x_k}^{-1}(x_{k+1}) = -\eta_k\Delta_k$, we have that
    \begin{align*}
    f(x_{k+1}) &\leq f(x_k) + \langle grad(f(x_k)), Exp_{x_k}^{-1}(x_{k+1}) \rangle  + \frac{L}{2} d^2(x_k, x_{k+1})  \\
    &= f(x_k) + \langle grad(f(x_k)), -\eta_k\Delta_k \rangle  + \frac{L}{2} \|-\eta_k\Delta_k\|^2\\
    & \leq f(x_k) +\langle grad(f(x_k)),-\eta_k grad(f(x_k))+\eta_k e_k\rangle \\
    & \quad +\frac{L\eta_k^2}{2}\left(2\|e_k\|^2+2\|grad(f(x_k))\|^2\right),
\end{align*}
where 
we used definition of $e_k$ and Young's Inequality. Now, using Young's Inequality again for any $\alpha_k>0$, we obtain:

\begin{align*}
    f(x_{k+1}) &\leq f(x_k) -\eta_k\|grad(f(x_k))\|^2+\frac{\eta_k^2}{2\alpha_k}\|grad(f(x_k))\|^2 +\frac{\alpha_k}{2}\|e_k\|^2 \\  
    & \quad +\frac{L\eta_k^2}{2}\left(2\|e_k\|^2+2\|grad(f(x_k))\|^2\right).
\end{align*}

Rearranging the terms, we get:

\begin{align}\label{gen_ineq}
    \left(\eta_k-\frac{\eta_k^2}{2\alpha_k}-L\eta_k^2\right)\|grad(f(x_k))\|^2 \leq f(x_k)-f(x_{K+1})+\left(\frac{\alpha_k}{2}+L\eta_k^2\right)\|e_k\|^2.
\end{align}

Choosing $\eta_k\leq \frac{\alpha_k}{1+2\alpha_k L}$, we can bound the left hand side from below by $\frac{\eta_k}{2}\|grad(f(x_k))\|^2$. 
\end{proof}

\begin{remark}[Extension to general retractions]\label{rem:retraction}
Although Algorithm~\ref{rigd} is stated using the exponential map, the same analysis extends to a general retraction \(R\) by replacing the update
\[
x_{k+1}=Exp_{x_k}(-\eta_k\Delta_k)
\quad \text{with} \quad
x_{k+1}=R_{x_k}(-\eta_k\Delta_k).
\]
In this case, it suffices to assume the retraction-smoothness condition in Assumption~\ref{asu1}, i.e.,
\[
f(R_x(z)) \le f(x)+\langle grad (f(x)),z\rangle_x+\frac{L}{2}\|z\|_x^2,
\]
for all admissible \(x\in\mathcal M\) and \(z\in T_x\mathcal M\). Under this assumption, the descent argument used in the proof proceeds by setting \(x=x_k\) and \(z=-\eta_k\Delta_k\), and all subsequent steps remain unchanged.
We retain the exponential map in the main presentation for clarity and because it is the canonical Riemannian update. Moreover, for the Grassmannian considered in our numerical experiments, \(Exp\) admits a closed-form expression, making it a natural choice in our setting.
\end{remark}

\begin{corollary}
    Define ${k^*}\in\mbox{argmin}_{0\leq k\leq K-1}\{{grad(f(x_k))}\}$. Under the premises of Theorem~\ref{th:rate}, the following hold:
    \begin{itemize}
    \item[(i)] If $\|e_k\|=\mathcal O(1/(k+1)^{0.5+\nu})$ for some $\nu>0$, $\alpha_k=1/L$ and $\eta_k=1/(3L)$, then  $\|grad(f(x_{k^*}))\|^2 = \mathcal O(1/K)$.
    \item[(ii)] If $\|e_k\|=\mathcal O(1/(k+1)^{0.5})$, $\alpha_k=\frac{1}{2L\log^2(k+2)}$ and $\eta_k=\alpha_k/2$, then $\|grad(f(x_{k^*}))\|^2\ = \mathcal O(\log^2(K+2)/K)$.
    \end{itemize}
\end{corollary}
\begin{proof}
(i) Defining  ${k^*}\in\mbox{argmin}_{0\leq k\leq K-1}\{{grad(f(x_k))}\}$, choosing $\alpha_k=1/L$ and $\eta_k=1/(3L)$ in \eqref{th:result} and summing over $k = 0, \dots, K-1$, we obtain:
\begin{align}
\label{cor:result1} \|grad(f(x_{k^*}))\|^2  \leq \frac{6L}{K}\left(f(x_0)-f(x^*)\right)+\frac{11}{3K}\sum_{k=0}^{K-1}\|e_k\|^2.
\end{align}
Now choosing $\|e_k\|=\mathcal O(1/(k+1)^{0.5+\nu})$ we can obtain
    \begin{align*}
        \sum_{k=0}^{K-1}\|e_k\|^2&=\sum_{k=0}^{K-1}\frac{1}{(k+1)^{1+2\nu}} \\
& =
\sum_{n=1}^{K}\frac{1}{n^{1+2\nu}} 
=
1+\sum_{n=2}^{K}\frac{1}{n^{1+2\nu}} 
\le
1+\int_{1}^{K} x^{-(1+2\nu)}\,dx 
\\&=
1+\frac{1-K^{-2\nu}}{2\nu}
< 1+\frac{1}{2\nu}.
    \end{align*}
    Hence, from \eqref{cor:result1} we get that
    \begin{align*}
       \|grad(f(x_{k^*}))\|^2\leq \frac{6L}{K}(f(x_0) - f(x^*)) + \frac{11}{3K}(1+\tfrac{1}{2\nu})= \mathcal O(1/K).
    \end{align*}
(ii) Choosing $\alpha_k=\frac{1}{2L\log^2(k+2)}$ and $\eta_k=\alpha_k/2$ in \eqref{th:result}, summing both sides over $k = 0, \dots, K-1$, and dividing by $K$, we obtain:
\begin{align*}
    &\frac{1}{K}\sum_{k=0}^{K-1}\frac{1}{8L\log^2(k+2)}\|grad(f(x_{k}))\|^2 \\&\quad\leq\frac{f(x_0)-f(x^*)}{K} + \frac{1}{K}\sum_{k=0}^{K-1}\left(\frac{1}{4L\log^2(k+2)}+\frac{1}{16L\log^4(k+2)}\right)\|e_k\|^2.
\end{align*}
Note that the $\frac{1}{\log^2(K+1)}\leq \frac{1}{\log^2(k+1)}$, hence we can bound the left hand side from below by $\frac{1}{8L\log^2(K+2)}\|grad(f(x_{k^*}))\|^2$, where $k^*$ such that ${k^*}\in\operatorname{argmin}_{0\leq k\leq K-1}\{{grad(f(x_k))}\}$.

Moreover, choosing $\|e_k\|^2=1/(k+1)$, one can show that:
\begin{align*}
    \sum_{k=0}^{K-1}\frac{1}{\log^2(k+2)}\|e_k\|^2\leq 4,\qquad
    \sum_{k=0}^{K-1}\frac{1}{\log^4(k+2)}\|e_k\|^2\leq 5.
\end{align*}
Hence, we obtain the following: 
\begin{align*}
   \|grad(f(x_{k^*}))\|^2 \leq \frac{8L\log^2(K+2)}{K}\left(f(x_0)-f(x^*)\right) + \frac{21}{2K}log^2(K+2).
\end{align*}
\end{proof}

{\bf Stochastic biased extension.} Although our main focus is on deterministic inexact Riemannian gradients, the same analysis also yields an extension to stochastic biased estimators. In particular, we consider the stochastic optimization problem
\[
\min_{x\in\mathcal M} f(x),
\quad \text{where }\quad f(x):=\mathbb E_{\omega}\big[f(x;\omega)\big],
\]
and \(f(x;\omega)\) denotes the sample objective function associated with the random variable \(\omega\). We further assume that the stochastic gradient estimator satisfies the following conditional bias--variance model.
\begin{assumption}\label{asu:stoch}
Let \(\{\omega_k\}_{k\ge 0}\) be the random variables used to generate the stochastic estimator \(\Delta_k\), and let $\{\mathcal{F}_k\}_{k\ge 0}$ be the filtration generated by $\{x_0,\omega_0,\ldots,\omega_{k-1}\}$.
Assume that $\Delta_k\in \mathcal{T}_{x_k}\mathcal{M}$ satisfies
\[
\mathbb{E}[\Delta_k \mid \mathcal{F}_k] = grad(f(x_k)) + b_k,
\]
where $b_k\in \mathcal{T}_{x_k}\mathcal{M}$ is a (possibly nonzero) bias term.
Moreover, the conditional variance of the stochastic fluctuation is bounded:
\[
\mathbb{E}\!\left[\left\|\Delta_k-\mathbb{E}[\Delta_k\mid \mathcal{F}_k]\right\|^2 \,\middle|\, \mathcal{F}_k\right]\le \sigma_k^2.
\]
\end{assumption}
Based on the stochastic biased estimator described in Assumption~\ref{asu:stoch}, we consider the following algorithm.
\begin{algorithm}[htbp]
\caption{Stochastic Biased Riemannian Gradient Descent (SB-RGD)}\label{alg:stoch}
\begin{algorithmic}[1]
\STATE{\bf Input:} Initial point $x_0 \in \mathcal{M}$, step sizes $\{\eta_k\}_{k\geq 0}$
\FOR{$k = 0, \dots, K-1$}  
\STATE Sample $\omega_k$ and compute $\widetilde{grad}(f(x_k;\omega_k))$
\STATE $\Delta_k \leftarrow \arg\min_{\Delta \in \mathcal{T}_{x_k}\mathcal{M}} \left\Vert \Delta - \widetilde{grad}(f(x_k;\omega_k)) \right\Vert$
\STATE $x_{k+1} = Exp_{x_k}(-\eta_k\Delta_k)$
\ENDFOR
\end{algorithmic}
\end{algorithm}
The next result shows that the deterministic descent estimate extends directly to this stochastic biased setting.
\begin{theorem}\label{th:rate_stoch}
Suppose Assumptions~\ref{asu1}--\ref{asu2} and Assumption~\ref{asu:stoch} hold.
Define the random error term $e_k \triangleq grad(f(x_k))-\Delta_k$.
Let $\{x_k\}$ be generated by Algorithm~\ref{alg:stoch}, with step size satisfying
$\eta_k \leq \frac{\alpha_k}{1 + 2\alpha_k L}$ for some $\alpha_k>0$.
Then, the following holds:
\begin{align}\label{th:result_stoch}
\mathbb{E}\!\left[\frac{\eta_k}{2}\|grad(f(x_k))\|^2\right]
\leq
\mathbb{E}\!\left[f(x_k)-f(x_{k+1})\right]
+\left(\frac{\alpha_k}{2}+L\eta_k^2\right)\mathbb{E}\!\left[\|b_k\|^2+\sigma_k^2\right].
\end{align}
(ii) Suppose $\eta_k=\eta$ and $\alpha_k=\alpha,$
with $\eta \le \frac{\alpha}{1+2\alpha L}$. If $\mathbb{E}[\|b_k\|^2+\sigma_k^2]=O(1/k^{1+\bar\delta})$ for some \(\bar\delta>0\), then \(\mathbb E\left[\|grad(f(x_{k^*}))\|^2\right] = \mathcal O(1/K)\).
\end{theorem}
\begin{proof} We can follow the steps of proof of Theorem \ref{th:rate} to show inequality \eqref{gen_ineq} as follows:
    \begin{align*}
\left(\eta_k-\frac{\eta_k^2}{2\alpha_k}-L\eta_k^2\right)\|grad(f(x_k))\|^2
\le
f(x_k)-f(x_{k+1})
+\left(\frac{\alpha_k}{2}+L\eta_k^2\right)\|e_k\|^2.
\end{align*}
Now take conditional expectation with respect to $\mathcal{F}_k$ on both sides.

\begin{align}\label{bound_cond_stoch}
\left(\eta_k-\frac{\eta_k^2}{2\alpha_k}-L\eta_k^2\right)\|grad(f(x_k))\|^2
& \le
\mathbb{E}\!\left[f(x_k)-f(x_{k+1}) \,\middle|\, \mathcal{F}_k\right] \nonumber \\
& \qquad +\left(\frac{\alpha_k}{2}+L\eta_k^2\right)
\mathbb{E}\!\left[\|e_k\|^2 \,\middle|\, \mathcal{F}_k\right].
\end{align}
Moreover, under Assumption~\ref{asu:stoch},
\[
e_k = grad(f(x_k))-\Delta_k
= -b_k - \left(\Delta_k-\mathbb{E}[\Delta_k\mid \mathcal{F}_k]\right),
\]
and the cross term vanishes under conditional expectation, giving
\[
\mathbb{E}\!\left[\|e_k\|^2 \,\middle|\, \mathcal{F}_k\right]
=
\|b_k\|^2+
\mathbb{E}\!\left[\left\|\Delta_k-\mathbb{E}[\Delta_k\mid \mathcal{F}_k]\right\|^2 \,\middle|\, \mathcal{F}_k\right]
\le \|b_k\|^2+\sigma_k^2.
\]
Choosing $\eta_k\le \frac{\alpha_k}{1+2\alpha_k L}$ lower bounds the coefficient on the left of \eqref{bound_cond_stoch} by $\eta_k/2$, substituting and taking expectation completes the proof.\\
(ii) Summing \eqref{th:result_stoch} from $k=0$ to $K-1$ gives
\begin{align*}
&\frac{\eta}{2}\sum_{k=0}^{K-1}\mathbb{E}\!\left[\|grad(f(x_k))\|^2\right] \\
& \qquad \le
\sum_{k=0}^{K-1}\mathbb{E}\!\left[f(x_k)-f(x_{k+1})\right]
+
\left(\frac{\alpha}{2}+L\eta^2\right)\sum_{k=0}^{K-1}\mathbb{E}\!\left[\|b_k\|^2+\sigma_k^2\right]\\
& \qquad =
\mathbb{E}[f(x_0)-f(x_K)]
+
\left(\frac{\alpha}{2}+L\eta^2\right)\sum_{k=0}^{K-1}\mathbb{E}\!\left[\|b_k\|^2+\sigma_k^2\right].
\end{align*}
Since $f(x_K)\ge f(x^\star)$, we obtain
\[
\frac{\eta}{2}\sum_{k=0}^{K-1}\mathbb{E}\!\left[\|grad(f(x_k))\|^2\right]
\le
f(x_0)-f(x^\star)
+
\left(\frac{\alpha}{2}+L\eta^2\right)\sum_{k=0}^{K-1}\mathbb{E}\!\left[\|b_k\|^2+\sigma_k^2\right].
\]
By assumption $\mathbb{E}\!\left[\|b_k\|^2+\sigma_k^2\right]=O\!\left(\frac{1}{k^{1+\bar\delta}}\right),$ for some $\bar\delta>0,$ we get that
$\sum_{k=0}^{K-1}\mathbb{E}\!\left[\|grad(f(x_k))\|^2\right]\le C'$
for some constant $C'>0$.
Then
\[
\mathbb{E}\!\left[\|grad(f(x_{k^*}))\|^2\right]\le \frac{C'}{K} = O\!\left(\frac{1}{K}\right),
\]
where ${k^*}\in\mbox{argmin}_{0\leq k\leq K-1}\{\mathbb E[{grad(f(x_k))}]\}$.
\end{proof}

\subsection{Boundedness of the Iterates Generated by RiGD}
So far, we have established convergence guarantees for the proposed inexact gradient method under the assumption that the error of estimating the gradient diminishes over iterations. However, depending on how this error enters the gradient computation, the analysis may implicitly rely on the boundedness of the iterates. Consider the quadratic objective function $f(x) = \frac{1}{2} x^\top A x + b^\top x + c$, where the gradient is given by  $\nabla f(x) = Ax + b$. Suppose we only have access to an approximate matrix $\tilde A =A + E$, where the error $E$ diminishes over time (i.e., $\|E_k\| \to 0$). In this case, the computed gradient becomes $ \tilde A x + b = \nabla f(x) + E x$, and hence the inexactness in the gradient is $e(x) = E x$.
While the perturbation $E$ may vanish, the total gradient error vanishes only if the iterates $x$ remain bounded. Therefore, to ensure the overall error in the inexact gradient method diminishes, it becomes essential to also establish that the iterates stay within a bounded region. This motivates the next part of our analysis, where we prove the boundedness of iterates generated by the proposed method.

To proceed with the boundedness analysis, we adopt a commonly used relative error assumption on the gradient approximation, namely:
\begin{align*}
    \| e(x) \| \leq \delta \| \nabla f(x) \|,
\end{align*} 
for some $\delta \in[0,1)$. This condition captures scenarios in which the inexactness in the gradient estimate diminishes proportionally as we approach stationarity, and it has been widely used in the analysis of inexact methods in Euclidean space~\cite{khanh2023inexact,vernimmen2024convergence, carter1991,khanh2024newinexactgradientdescent}. Our assumption aligns with these frameworks and facilitates the derivation of boundedness results in the manifold setting. Next, we formally state our assumption.

\begin{assumption}\label{assump:relative}(Relative Error Condition)
Let $e_k \triangleq grad(f(x_k)) - \Delta_k$ denote the gradient approximation error at iteration $k$. There exists a constant $\delta \in [0,1)$ such that
$$\|e_k\| \leq \delta \|grad(f(x_k))\| \quad \text{for all } k.$$
\end{assumption}

In the next theorem, we show that the sequence of iterates generated by Algorithm \ref{rigd} is bounded.
\begin{theorem}\label{th:bounded}(Boundedness of Iterates)
    Suppose Assumptions \ref{asu1}-\ref{assump:relative} hold, then  (i) the sequence $\{x_k\}_k$ generated by Algorithm \ref{rigd} is bounded; (ii) Choosing $\eta_k=\frac{1}{L}$, then $\|grad(f(x_{k^*}))\|^2 = \mathcal O(1/K)$, where ${k^*}\in\mbox{argmin}_k\{{grad(f(x_k))}\}$.
\end{theorem}
\begin{proof}
We prove the result using induction by showing that for any $k\geq 0$, $x_k\in \mathcal{L}_{f(x_0)}(f)$, where $\mathcal{L}_{c}(f)\triangleq \{x\in\mathcal{M}\mid f(x)\leq c\}$ denotes the sublevel set of the function $f$ at some $c\in\mathbb R$. First, note that $x_0\in \mathcal{L}_{f(x_0)}(f)$ trivially holds. Now, suppose $x_k\in \mathcal{L}_{f(x_0)}(f)$. 
Using Assumption \ref{asu1} we have:

\begin{align*}
    f(x_{k+1}) &\leq f(x_k) + \langle grad(f(x_k)), Exp_{x_k}^{-1}(x_{k+1}) \rangle  + \frac{L}{2} d^2(x_k, x_{k+1})  \\
    &= f(x_k) + \langle grad(f(x_k)), -\eta_k\Delta_k \rangle  + \frac{L}{2} \|-\eta_k\Delta_k\|^2\\
    &= f(x_k) +\langle grad(f(x_k)),-\eta_k grad(f(x_k))+\eta_k e_k\rangle \\
    & \quad + \frac{L\eta_k^2}{2}\left(\|e_k\|^2+\|grad(f(x_k))\|^2\right) -L\eta_k^2\langle e_k,grad(f(x_k))\rangle\\
    &=f(x_k)-\left(\eta_k-\frac{L\eta_k^2}{2}\right)\|grad(f(x_k))\|^2 \\
    & \quad +\left(\eta_k-L\eta_k^2\right)\langle e_k,grad(f(x_k))\rangle+\frac{L\eta_k^2}{2}\|e_k\|^2,
\end{align*}
where we used $Exp_{x_k}^{-1}(x_{k+1}) = -\eta_k\Delta_k$ and $d(x_k, x_{k+1}) = \left\Vert -\eta_k\Delta_k \right\Vert$. Moreover, from Assumption \ref{assump:relative}, we have that $$\langle e_k,grad(f(x_k))\rangle\leq \|e_k\|\|grad(f(x_k))\|\leq \delta\|grad(f(x_k))\|^2.$$ Hence, by choosing $\eta_k\leq 1/L$ and simplifying the terms, we get
\begin{align}\label{bound f}
    f(x_{k+1}) \leq f(x_k) -\eta_k(1-\delta-L\eta_k((1-\delta)^2/2))\|grad(f(x_k))\|^2.
\end{align}

Since $\eta_k\leq \frac{1}{L} < \frac{2}{L(1-\delta)}$, one can confirm that $f(x_{k+1}) \leq f(x_k)$. Therefore, using the induction's assumption we conclude that $f(x_{k+1})\leq f(x_0)$, hence, $x_{k+1}\in \mathcal{L}_{f(x_0)}(f)$. Finally, since the level set of function $f$ is bounded, we conclude that the sequence of iterates generated by Algorithm \ref{rigd} is bounded.

(ii) Choose $\eta_k=\eta=\frac{1}{L}$, then from \eqref{bound f}, we get:

\begin{align*}
    \frac{\eta(1-\delta^2)}{2}\|grad(f(x_k))\|^2\leq f(x_k)-f(x_{k-1}).
\end{align*}

Now summing both sides over $k=0,\hdots,K-1$ and dividing by $K$, we obtain:

\begin{align*}
    \|grad(f(x_{k^*}))\|^2\leq \frac{2(f(x_0)-f(x^*))}{K\eta(1-\delta^2)},
\end{align*}
where ${k^*}\in\mbox{argmin}_k\{{grad(f(x_k))}\}$.
\end{proof}

\section{Adaptive Step Size via Line Search}\label{sec:line search}

In this section, we extend our inexact Riemannian gradient method to incorporate an adaptive step size strategy based on a backtracking line search. This modification ensures sufficient decrease in the objective function at each iteration, while allowing the method to adjust to local geometry and gradient accuracy without requiring prior knowledge of the smoothness constant. Unlike fixed or diminishing step size schemes, the line search mechanism dynamically selects the step size $\eta_k$ to satisfy a sufficient decrease condition, even in the presence of inexact gradient evaluations.

This approach relies on the assumption that while the gradient of the objective function is only available inexactly, the function value itself can be evaluated exactly. More specifically, at each iteration $k$, given an inexact gradient $\Delta_k$, we determine $\eta_k$ via backtracking to ensure that $x_{k+1}$ satisfies
\begin{equation}\label{eq:suff_decrease}
f(x_{k+1}) \leq f(x_k) - \sigma \eta_k \|\Delta_k\|^2,
\end{equation}
for some fixed parameter $\sigma \in (0,1)$. This condition guarantees that the algorithm makes measurable progress toward stationarity (See Algorithm \ref{alg2}). To ensure the termination of the line search, we consider a relative bound on the gradient error (also defined in Assumption \ref{assump:relative}) of the form
$$\|e_k\| \leq \delta \| {grad}(f(x_k))\|, \quad \text{for some } \delta \in [0,1),$$
ensuring that the descent direction remains sufficiently aligned with the true gradient. Under these conditions, we show that the line search terminates in finite steps and prove that the algorithm retains an $\mathcal{O}(1/K)$ convergence rate in terms of the squared norm of the Riemannian gradient.
\begin{algorithm}[htbp]
\caption{Riemannian Inexact Gradient Descent with Line Search (RiGD-LS)}\label{alg2}
\begin{algorithmic}[1]
\STATE{\bf Input:} Initial point $x_0 \in \mathcal{M}$, initial step size $\eta_0 > 0$, shrinkage factor $\beta \in (0,1)$, sufficient decrease parameter $\sigma \in (0,1)$
\FOR{$k = 0, \dots, K-1$}  
    \STATE $\Delta_k \leftarrow \arg\min_{\Delta \in \mathcal{T}_{x_k} \mathcal{M}} \left\Vert \Delta - \widetilde{grad}(f(x_k)) \right\Vert$
    \STATE  $\eta_k \leftarrow \eta_0$
    \WHILE{$f(Exp_{x_k}(-\eta_k \Delta_k)) > f(x_k) - \sigma \eta_k \|\Delta_k\|^2$}
        \STATE $\eta_k \leftarrow \beta \cdot \eta_k$
    \ENDWHILE
    \STATE $x_{k+1} \leftarrow Exp_{x_k}(-\eta_k \Delta_k)$
\ENDFOR
\end{algorithmic}
\end{algorithm}

In the next theorem, we show that the backtracking line search terminates in finitely many steps by identifying a threshold step size $\bar{\eta}$ such that the sufficient decrease condition,
    $f(x_{k+1}) \leq f(x_k) - \sigma \eta_k \|\Delta_k\|^2$,
is satisfied whenever $\eta_k \leq \bar{\eta}$. 

\begin{theorem}\label{th:terminate}(Finiteness of line search)
    Suppose Assumptions \ref{asu1} and \ref{assump:relative} hold. Let $\bar\eta := \frac{2((1-\delta)-\sigma(1+\delta)^2)}{L(1+\delta)^2}$. For any $\eta_k \in (0, \bar\eta]$, the following holds, 
    \begin{align*}
        f(x_{k+1}) &\leq f(x_k) - \sigma \eta_k \|\Delta_k\|^2. 
    \end{align*}
\end{theorem}
\begin{proof}From smoothness of $f$, we have that:
\begin{align*}
    f(x_{k+1}) \leq f(x_k) + \langle   {grad}(f(x_k)),  {Exp}_{x_k}^{-1}(x_{k+1}) \rangle + \frac{L}{2} d^2(x_k, x_{k+1}).
\end{align*}

Using the fact that ${Exp}_{x_k}^{-1}(x_{k+1}) = -\eta_k \Delta_k,$
$d(x_k, x_{k+1}) = \eta_k \|\Delta_k\|$,
we obtain
\begin{align*}
    f(x_{k+1}) \leq f(x_k) - \eta_k \langle   {grad}(f(x_k)), \Delta_k \rangle + \frac{L}{2} \eta_k^2 \|\Delta_k\|^2.
\end{align*}

Now, we bound the inner product. By the definition of $\Delta_k$ and Cauchy–Schwarz:
\begin{align*}
    \langle   {grad}(f(x_k)), \Delta_k \rangle 
    &= \|  {grad}(f(x_k))\|^2 - \langle   {grad}(f(x_k)), e_k \rangle \\
    &\geq \|  {grad}(f(x_k))\|^2 - \|  {grad}(f(x_k))\| \cdot \|e_k\| \\
    &\geq (1 - \delta) \|  {grad}(f(x_k))\|^2,
\end{align*}
where we used the assumption $\|e_k\| \leq \delta \|  {grad}(f(x_k))\|$ for some $\delta \in [0,1)$. 
Similarly, we bound the norm of the inexact direction:
\begin{align*}
    \|\Delta_k\| &\leq \|  {grad}(f(x_k))\| + \|e_k\| 
    \leq (1 + \delta) \|  {grad}(f(x_k))\|.
\end{align*}

Combining these bounds, we obtain:
\begin{align*}
    f(x_{k+1}) 
    \leq f(x_k) - \eta_k (1 - \delta) \|  {grad}(f(x_k))\|^2 
    + \frac{L}{2} \eta_k^2 (1 + \delta)^2 \|  {grad}(f(x_k))\|^2.
\end{align*}

Choosing $\eta_k\leq \frac{2((1-\delta)-\sigma(1+\delta)^2)}{L(1+\delta)^2}:=\bar \eta$, we can bound the right-hand-side of the above inequality as follows
\begin{align*}
    f(x_{k+1}) 
    \leq f(x_k)-\sigma\eta_k(1+\delta)^2 \|grad(f(x_k))\|^2.
\end{align*}
Moreover, using $\|e_k\|\leq \delta\|grad(f(x_k))\|$ and $\|e_k\|\geq \|\Delta_k\|-\|grad(f(x_k))\|$, we obtain that
$$f(x_{k+1})\leq f(x_k)-\sigma\eta_k\|\Delta_k\|^2,$$
which completes the proof.
\end{proof}

\begin{theorem}
    Suppose Assumptions \ref{asu1}, \ref{asu2} and \ref{assump:relative} hold. Let $\{x_k\}$ be the sequence of iterates generated by Algorithm \ref{alg2} and $\{\eta_k\}$ the sequence of step sizes generated via backtracking line search. Then, the following holds
    \begin{align*}
        \|grad(f(x_{k^{*}}))\|^2\leq \frac{f(x_0)-f(x_K)}{K\sigma\eta_{\min}\left(1-\frac{1}{\bar\alpha}+(1-\bar\alpha)\delta^2\right)},
    \end{align*}
    where $\bar\alpha \in (1, 1/\delta^2)$, $x_{k^{*}}\in\mbox{argmin}_k\{{grad(f(x_k))}\}$, and $\eta_{\min}$ is defined in Theorem \ref{th:terminate}.
\end{theorem}
\begin{proof}
From definition of $\Delta_k$, we have that:
\begin{align*}
\|\Delta_k\|^2=\|grad(f(x_k))\|^2+\|e_k\|^2-2\langle grad(f(x_k)),e_k\rangle.
\end{align*}

Using Young's inequality we have that $2\langle grad(f(x_k)),e_k\rangle\leq \frac{1}{\bar\alpha}\|grad (f(x_k))\|^2+\bar\alpha\|e_k\|^2,$ for any $\bar\alpha>1$. Hence, we have that
\begin{align*}
\|\Delta_k\|^2\geq\left(1-\frac{1}{\bar\alpha}\right)\|grad(f(x_k))\|^2+(1-\bar\alpha)\|e_k\|^2.
\end{align*}

Using the assumption that $\|e_k\|\leq \delta \|grad(f(x_k))\|$ and $1 - \bar\alpha < 0$, we get:
\begin{align}\label{bound delta2}
    \|\Delta_k\|^2\geq \left(1-\frac{1}{\bar\alpha}+(1-\bar\alpha)\delta^2\right)\|grad(f(x_k))\|^2.
\end{align}

From the backtracking step, we have that $f(x_{k+1})\leq f(x_k)-\sigma\eta_k\|\Delta_k\|^2$, hence using \eqref{bound delta2} we get:
\begin{align}\label{f bound linesearch}
    f(x_{k+1})\leq f(x_k)-\sigma\eta_k\left(1-\frac{1}{\bar\alpha}+(1-\bar\alpha)\delta^2\right)\|grad(f(x_k))\|^2.
\end{align}

Rearranging the terms, summing \eqref{f bound linesearch} over $k=0,\hdots,K-1$, dividing both sides by $K$, and choosing $\bar\alpha\in(1,1/\delta^2)$, one can obtain:
\begin{align*}
    \frac{1}{K}\sum_{k=0}^{K-1}\|grad(f(x_k))\|^2\leq \frac{f(x_0)-f(x_K)}{K\sigma\eta_{\min}\left(1-\frac{1}{\bar\alpha}+(1-\bar\alpha)\delta^2\right)},
\end{align*}
where $\eta_{\min}$ is defined in Theorem \ref{th:terminate}. Now, choosing $x_{k^{*}}\in\mbox{argmin}_k\{{grad(f(x_k))}\}$, we get the desired result.
\end{proof}

\subsection{Boundedness of the Iterates Generated by RiGD-LS} As shown in inequality \eqref{f bound linesearch}, when $\bar{\alpha} \in (1, 1/\delta^2)$, which ensures that $(1-\frac{1}{\bar\alpha}+(1-\bar\alpha)\delta^2)$ is positive which means that $f(x_{k+1}) \leq f(x_k)$ for all $k$, hence the sequence $\{f(x_k)\}$ is not increasing. Therefore, along the same lines of our argument in the proof of Theorem \ref{th:bounded}, the iterates generated by Algorithm \ref{alg2} remain bounded.

\section{RiGD on the Grassmannian}\label{sec:grassmann}

In this section, we focus on the Grassmann manifold. The collection of all linear subspaces of dimension $p$ of $\mathbb{R}^n$ forms the Grassmann manifold denoted by $\mathrm{Gr}(p,n)$ where $p < n$ and $p,n \in \mathbb{Z^+}$. Points on the manifold are equivalence classes of $p \times n$ orthonormal matrices denoted by $[X]$. Two matrices are equivalent if they are related by right multiplication of an orthogonal $p \times p$ matrix \cite{edelman1998}. For any point $X_{p \times n}$ on the manifold and orthogonal matrix $M_{p \times p}$, the matrix product $M_{p \times p}X_{p \times n}$ is a point on the manifold since $X, MX \in [X]$. For the Grassmann manifold, the tangent space at a point $[X]$ is the set of all $p \times n$ matrices $A$ such that \[A X^\top = 0.\] 


For a given function $f(X)$ defined on the Grassmann manifold, the gradient of $f$ at $[X]$ can be calculated as \[grad(f(X)) = \frac{\delta f}{\delta X_{ij}} - XX^\top \frac{\delta f}{\delta X_{ij}},\] where $\delta f / \delta X_{ij}$ is a $p \times n$ matrix of partial derivatives of $f$ with respect to the elements of $X$. Additionally, the canonical metric for the Grassmann manifold can be defined as follows for any $p \times n$ matrices $A_1$ and $A_2$ such that $X^\top A_i = 0$ for ($i = 1,2$) \cite{edelman1998},
\begin{align*}
    g_c(A_1,A_2) = tr(A_1^\top(I - X^\top X)A_2) = tr(A_1^\top A_2).
\end{align*}

Now, the subproblem in step 3 of Algorithm \ref{rigd} and \ref{alg2} is a projection onto the tangent space of the Grassmann manifold. The projection operator onto the Grassmannian tangent space has a closed form expression. Given our estimated gradient $\widetilde{grad}$, the projection of $\widetilde{grad}$ onto the tangent space $\mathcal{T}_{x_k}\mathrm{Gr}(p,n)$ is given by
\begin{align*}
    P_{\mathcal{T}_{x_k}\mathrm{Gr}(p,n)}(\widetilde{grad}(f(x_k))) = \widetilde{grad}(f(x_k))(I - x_k^\top x_k),
\end{align*}
where $x_k$ is a point on the Grassmannian at iteration $k$, $I_{p \times p}$ denotes the identity and $\widetilde{grad}$ is the inexact gradient of $f$ at $x_k$. Another key ingredient of the algorithm is the exponential map (see Definition \ref{def1}). For a given point $x_{k}$ on the Grassmann manifold, the exponential map starting at $x_k$ in the direction of $\Delta$ is given by 
\begin{align*}
    Exp_{x_{k}}(\eta\Delta) = x_k U cos(\eta\Sigma) U^\top + V sin(\eta\Sigma) U^\top,
\end{align*}
where $U\Sigma V^\top$ is the singular value decomposition of $\Delta$. 
\begin{algorithm}[htbp]
\caption{Grassmannian Inexact Gradient Descent (GiGD)}\label{gigd}
\begin{algorithmic}[1]
\STATE{\bf Input:} Initial point $x_0 \in \mathrm{Gr}(p,n)$, step size $\eta$ 
\FOR{$k = 0, \dots, K-1$}  
\STATE  $\Delta_k \leftarrow \widetilde{grad}(f(x_k))(I - x_k^\top x_k)$ 
\STATE Compute the SVD: $\Delta_k = U_k \Sigma_k V_k^\top$
\STATE $x_{k+1} = x_k V_k cos(-\eta\Sigma_k) V_k^\top + U_k sin(-\eta\Sigma_k)V_k^\top$ 
\ENDFOR
\end{algorithmic}
\end{algorithm} 
We restate our algorithm with closed form expressions for the Grassmannian in Algorithm \ref{gigd}. 

\begin{algorithm}[htbp]
\caption{Grassmannian Inexact Gradient Descent with Line Search (GiGD-LS)}\label{gigd-ls}
\begin{algorithmic}[1]
\STATE{\bf Input:} Initial point $x_0 \in \mathrm{Gr}(p,n)$, initial step size $\eta_0 > 0$, shrinkage factor $\beta \in (0,1)$, sufficient decrease parameter $\sigma \in (0,1)$
\FOR{$k = 0, \dots, K-1$}  
    \STATE $\Delta_k \leftarrow \widetilde{grad}(f(x_k))(I - x_k^\top x_k)$
    \STATE Compute the SVD: $\Delta_k = U_k \Sigma_k V_k^\top$
    \STATE $\eta_k \leftarrow \eta_0$
    \WHILE{$f\left(x_k V_k \cos(-\eta_k \Sigma_k) V_k^\top + U_k \sin(-\eta_k \Sigma_k) V_k^\top\right) > f(x_k) - \sigma \eta_k \|\Delta_k\|^2$}
        \STATE $\eta_k \leftarrow \beta \cdot \eta_k$
    \ENDWHILE
    \STATE $x_{k+1} \leftarrow x_k V_k \cos(-\eta_k \Sigma_k) V_k^\top + U_k \sin(-\eta_k \Sigma_k) V_k^\top$
\ENDFOR
\end{algorithmic}
\end{algorithm}
Similar to the general Riemannian setting, we can incorporate a backtracking line search strategy into the Grassmannian Inexact Gradient Descent (GiGD) algorithm. The Grassmannian Inexact Gradient Descent with Line Search (GiGD-LS) algorithm, Algorithm \ref{gigd-ls}, extends the basic GiGD method by incorporating a backtracking line search strategy. At each iteration, we compute an inexact gradient and project it onto the tangent space of the Grassmann manifold. The resulting direction $\Delta_k$ is decomposed using the compact SVD to facilitate a geodesic update. The step size $\eta_k$ is adaptively selected by backtracking until a sufficient decrease condition is satisfied. The iterate is then updated using a closed-form expression for the exponential map on the Grassmannian, ensuring that each step remains on the manifold. 

Note that, the convergence guarantee for the Grassmannian variant, Algorithms \ref{gigd} and \ref{gigd-ls}, remains the same as that of the general Riemannian case, preserving the $\mathcal{O}(1/K)$ rate under same assumptions.
\section{Numerical Results on Imaging} \label{results}

\subsection{Problem formulation and bias in gradient estimates}

In this section we consider the optimization of channelized quadratic observers for binary classification discussed briefly in Section \ref{sec1}. Given $g$ an $n \times 1$ vector of measurements made by an imaging system and its channelized representation $v = Tg$, where $T$ is an $p \times n$ matrix ($p << n$), we aim to find the optimal $T$ to minimize detection error. As an optimization problem, we aim to maximize the given figure of merit by solving for the optimal channelizing matrix $T$, defined on the Grassmann manifold. This is a consequence of Eq. \eqref{inherent_grass}, where the function value must depend only on the subspace spanned by the rows of $T$ and not its specific realization. Thus, candidate solutions to our problem belong to an equivalence class making the problem inherently Grassmannian. Specifically, the problem takes the following form
\begin{align}\label{cqo_prob}
\max_{T \in \mathrm{Gr}(p,n)} J(T),
\end{align}
where $J$ is Jeffrey's divergence, also known as symmetrized Kullback-Leibler divergence. This divergence does not satisfy the triangle inequality and is thus not a metric in the true sense. However, symmetry is satisfied under this formulation. 

\begin{figure}[htbp]
    \centering
    \fbox{{\includegraphics[scale=0.5]{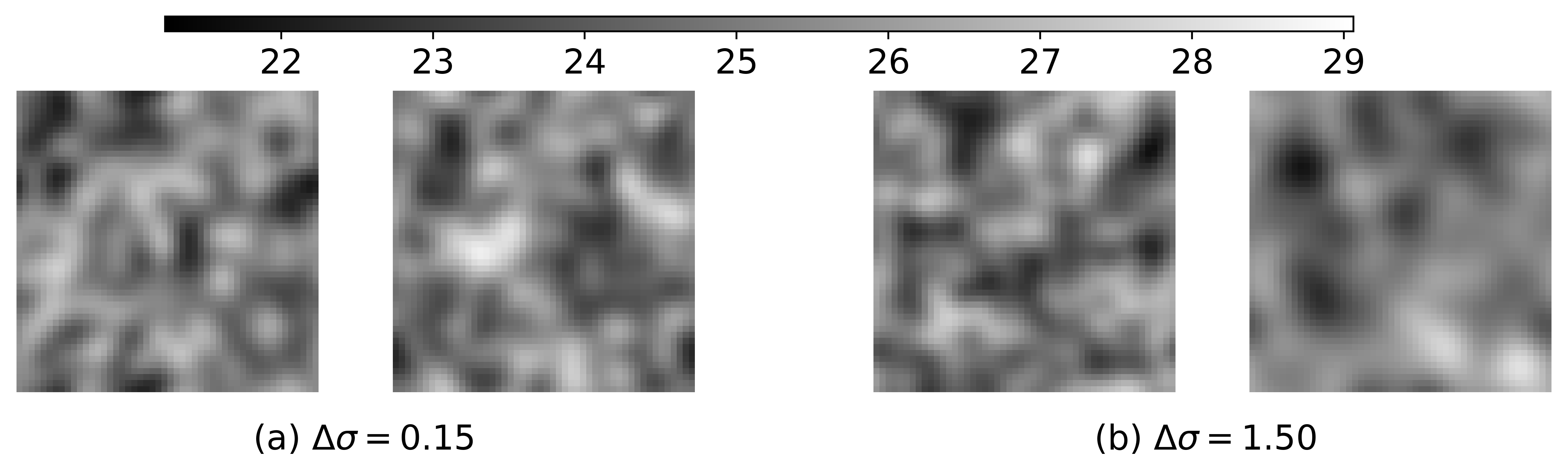}}}
    \caption{An example pair of sample images from each class. The images in (a) appear more similar due to the smaller correlation length difference, $\Delta\sigma = 0.15$ pixels, relative to those in (b)  where $\Delta\sigma = 1.50$ pixels.}
    \label{fig:sample_images}
\end{figure}

In practice, the dimension of the lexicographically ordered image, $n$, can be large, growing proportionally with the number of pixels in an image. This leads to computationally expensive matrix inversions necessary for computing pertinent figures of merit. The approach presented by the authors in \cite{kupinski_15} reduces the dimension of the problem leading to faster computation, with inversions performed on channelized matrices of sufficiently small dimension. This channelized approach ensures covariance matrix estimates are full rank in lower-dimensional space thus invertible - key to the computation of Jeffrey's divergence measure in lower-dimensional space. Now, Jeffrey's divergence $J: \mathrm{Gr}(p,n) \rightarrow \mathbb{R}_0^+$ under Gaussian statistics yields a closed form expression given by
\begin{align}\label{Jdiv}
    \nonumber J(T) &= -2L + tr(C_2^{-1}(T)C_1(T)) + s^\top T^\top C_2^{-1}(T) T s\\& \quad+ 
    tr(C_1^{-1}(T)C_2(T)) + s^\top T^\top C_1^{-1}(T) T s,  
\end{align} 
where is $s$ the difference of the means of the two images classes, $L$ a known scalar constant, $T_{p \times n}$ a point on the Grassmann manifold $\mathrm{Gr}(p,n)$, $C_i(T) = T K_i T^\top$ are the ($p \times p$) ``channelized" covariance matrices and the index $i = 1,2$ represents the class with $K_i$ representing the ($n \times n$) covariance matrices. One may view this approach as the reparameterization of the original problem of maximizing Jeffrey's divergence between two distributions in a higher dimensional space to a lower-dimensional space via the optimal matrix $T^*$ which maximizes the divergence in lower-dimensional space. Eq. \eqref{Jdiv} can take values in $[0, \infty)$ and an unconstrained maximization of this function is unbounded. However, the problem is constrained by the requirement that the rows of $T$ be linearly independent and the Grassmannian formulation of the optimization problem ensures these constraints are met. The gradient of the figure of merit from Eq. \eqref{Jdiv} is given by
\begin{equation}\label{eq5}
    \begin{split}
        grad(J(T))= C_1^{-1}(T)T(K_2 + ss^\top)[I - T^\top C_1^{-1}(T)TK_1] \\
             \quad + C_2^{-1}(T)T(K_1 + ss^\top)[I - T^\top C_2^{-1}(T)TK_2].
    \end{split}
\end{equation}

In a numerical study of heteroscedastic image data with correlation length $\Delta\sigma = \sigma_1 - \sigma_2 = 0.15$ pixels, $2500 \times 2500$ covariance matrices and $T$ with dimension $25 \times 2500$, we analyze the presence of bias in Eq. \eqref{Jdiv} and Eq. \eqref{eq5}. The covariance matrices for the sample images are chosen to be circulant with a mean pixel value of $25$. Each sample image of dimension $50 \times 50$ is generated from a multivariate Gaussian with circulant covariances using the affine transformation $X = \mu + AZ$ where $Z \sim N(0, I_n)$, $\mu = 25$ and $A$ is the matrix square root of the circulant covariance matrices. We refer the reader to \cite{kupinski_15} for additional details. Note that for a given covariance matrix $K_{n \times n}$ and sample size $N$, it is known that the sample covariance approaches the true covariance matrix at a rate of $O_p(n/\sqrt{N})$ (see \cite{puchkin2024sharperdimensionfreeboundsfrobenius}, \cite{vershynin2011introductionnonasymptoticanalysisrandom} Corollary 5.50). For fixed $n$, by the law of large numbers one can show that $|| K - \widetilde{K} ||_F = O_p(1/\sqrt{N})$ and the estimation error shrinks at the rate of $O(N^{-1/2})$. We confirm this behavior in Table \ref{tab:k1_k2_results}. From the theoretical perspective, this implies the errors in our problem are summable. 

\begin{table}[t] 
\captionsetup{width=\textwidth}
    \centering
    \caption{The Frobenius norm mean $\pm$ standard deviation compares the true and estimated covariance. As expected, both values decrease with sample size. Heteroscedastic images are simulated with $\sigma_1 = 3.0$ and $\sigma_2 = 4.5$ and equal means. $\tilde K_i$ denotes the sample covariance} 
    \label{tab:k1_k2_results} 
\begin{tabular}{rll} 
\toprule
Sample Size & $||K_1 - \widetilde{K}_1||_F$ & $||K_2 - \widetilde{K}_2||_F$ \\
\midrule
    10 & $844.9 \pm 42.7$ & $843.5 \pm 68.8$ \\
    100 & $252.3 \pm 4.5$ & $253.7 \pm 8.2$ \\
    1000 & $79.5 \pm 1.0$ & $79.9 \pm 2.0$ \\
    10000 & $25.1 \pm 0.3$ & $25.3 \pm 0.7$ \\
\bottomrule
\end{tabular}
\end{table}

When estimating the covariance matrix, gradient calculations for the figure of merit lead to an \textit{inexact gradient}. 
Covariance matrix estimation is particularly challenging under HDLSS settings. We note in Figure \ref{fig:sample_cov_jbias} (a) that the objective function Eq. \eqref{Jdiv} exhibits bias - Jeffrey's divergence for sample covariance matrices is \textit{larger} than for true covariance matrices $J(T,K) = 0.031$.  Figure \ref{fig:sample_cov_jbias} (b) plots the Frobenius norm of the difference of the gradient under true and sample covariance matrices evaluated at a random point on the Grassmannian. The element-wise mean gradient is computed and the Frobenius norm of the difference between the mean and true gradient is reported. Although the estimator is biased for finite sample sizes, the gradient approximation error decreases as the number of samples increases. Given that the norm of the difference is always positive, we can not infer the direction of the bias but since it is positive, we conclude the estimated gradient is biased.

\begin{figure*}[htbp]
    \centering
    \subfloat[Bias in J-Divergence values]{\includegraphics[width=0.49\textwidth]{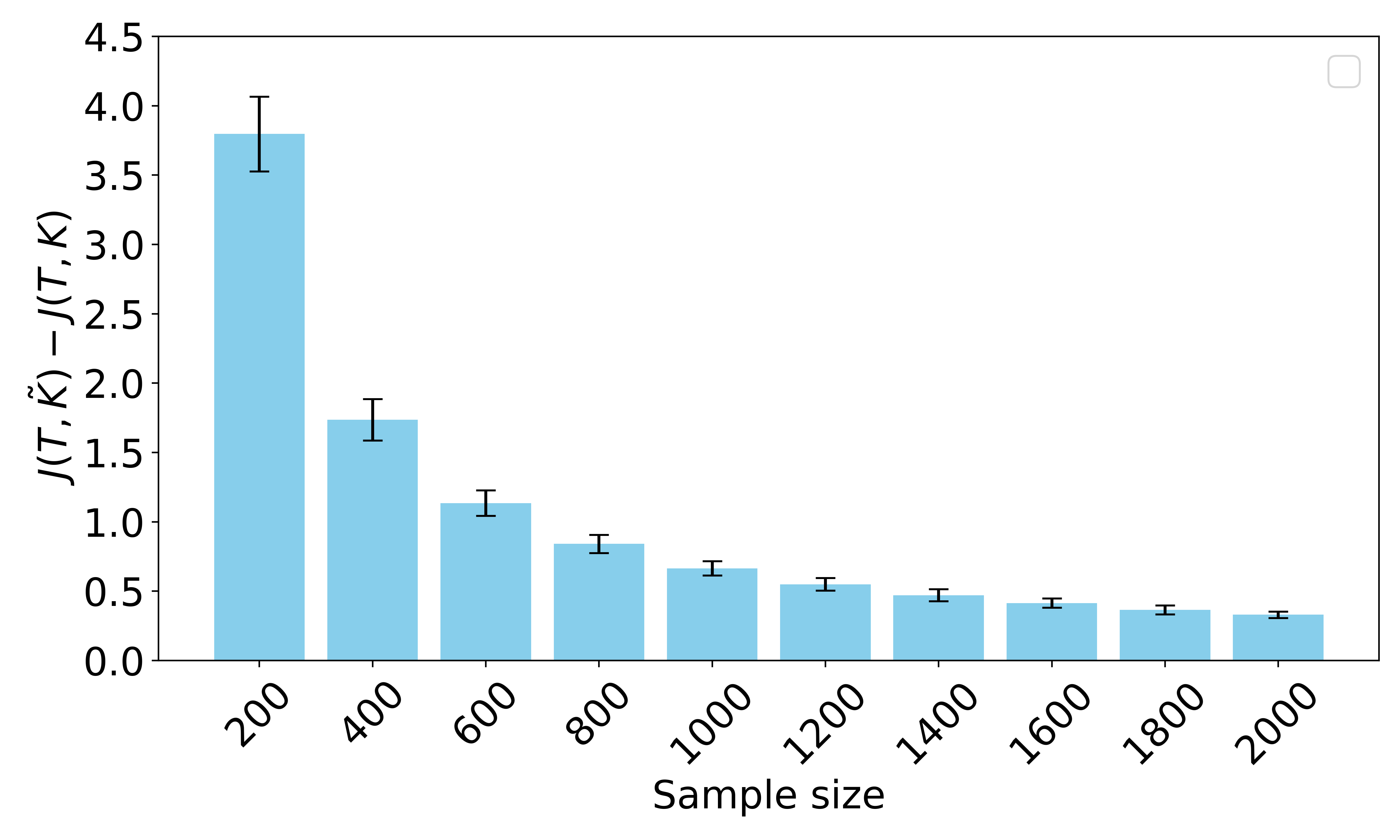}}
    \hspace{0.1mm}
    \subfloat[Gradient]{\includegraphics[width=0.48\textwidth]{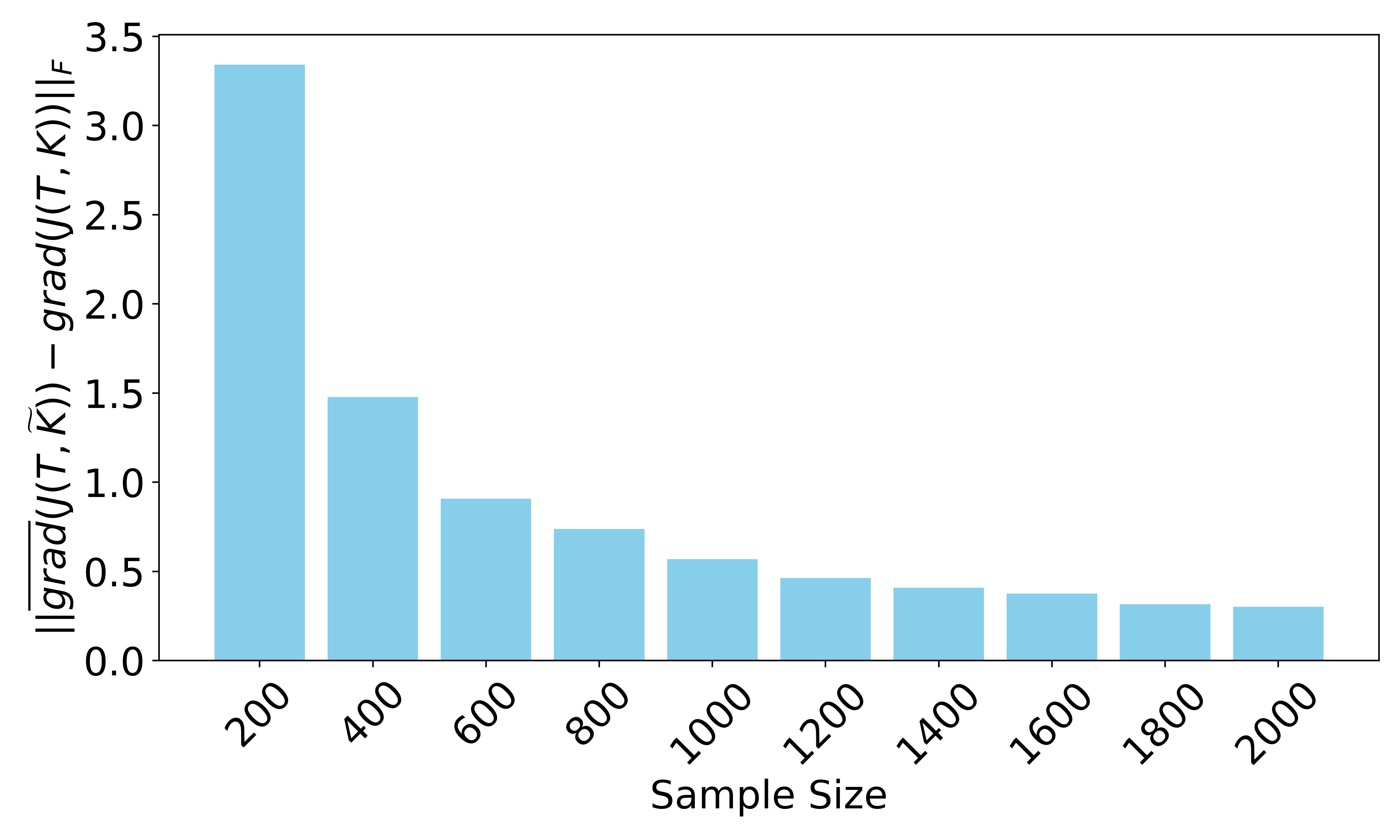}}
    
    \caption{The covariance in Eq. \eqref{cov_est} is estimated by an increasing quantity of samples to verify that both the objective function, Eq. \eqref{Jdiv}, in (a) and in (b) its gradient, Eq. \eqref{eq5}, approach the true value. In (a) the true J-divergence is subtracted from the mean J-divergence and is always positive, thus sample statistics are biased high. In (b), the Frobenius norm between the true and estimated mean gradient ($\overline{grad}$) demonstrate convergence with quantity of samples. For a given sample size, 100 replicates are used to estimate the mean (blue bar) and standard deviation (error bars).} 
    \label{fig:sample_cov_jbias}
\end{figure*}

Due to the presence of bias, traditional stochastic methods relying on unbiased estimators are not viable, necessitating the need for an inexact approach to the Grassmannian optimization problem. The proposed algorithm is a tool for solving a manifold based optimization problem with inexact gradients. Manifold-based dimensionality reduction is achieved via the decision variable $\bf{T}_{L \times M}$, a point on the Grassmannian, which reduces an $M-$dimensional covariance estimation and data discrimination problem to an $L-$dimensional one. In our numerical experiments we have $M = 2500$ and $L = 25$ - dimensionality reduction on the order of $100$ times.

\subsection{Implementation of optimization methods}

Having established the biased nature of the problem, two scenarios are analyzed using the RiGD algorithm. The first is an intentional perturbation of the covariance matrix where $U(0,1)$ random variates are added to each element of the covariance. For a given element of the covariance matrix, a $U(0,1)$ random variate (divided by the number of iterations) is added. In the limit as $k \rightarrow \infty$, the $(i,j)$th element of $K_1$ and $K_2$ will tend to its true value. At iteration $k$, the $(i,j)$th element of the covariance matrix is perturbed as 
\begin{align}\label{unif_bias}
    \hat{K}_{i,j} = K_{i,j} + \frac{Y_k}{200k},
\end{align}
where $Y_k \sim U(0,1)$. In the figures that follow, we refer to this uniform perturbation as $\hat K$ for brevity. For a finite number of iterations, the perturbed gradient will be biased. In the second scenario, RiGD is used to maximize Eq. \eqref{Jdiv} using sample covariance matrices estimated from a set of simulated images from both classes. The bias in this setting arises from covariance matrix estimation which is computed as
\begin{align*}
    \tilde{K}_{j,k} = \frac{1}{N-1} \sum_{i=1}^N (g_i[j] - \bar{g}[j])(g_i[k] - \bar{g}[k]).
\end{align*}Alternatively, in matrix-vector notation as 
\begin{align}\label{cov_est}
    \tilde{K} = \frac{1}{N-1} (G - 1\mu)^\top (G - 1\mu),
\end{align}
where $\tilde{K}_{j,k}$ is the sample covariance between the $j^{th}$ and $k^{th}$ pixel values estimated using images $g_i, \quad i \in \{1,\dots,N\}$. The $j^{th}$ pixel value of image $g_i$ is denoted by $g_i[j]$ and similarly the $k^{th}$ by $g_i[k]$. The variable $\bar{g}[j]$ is the mean of the $j^{th}$ pixel value calculated across all $N$ images. In matrix-vector notation, \textbf{G} denotes the matrix of sample images, with each row corresponding to an image, $\mu$ is the mean pixel-value vector and $1$ a $p \times 1$ vector. Note that the estimation of a covariance matrix is a non-trivial problem and there is a large body of research tackling this challenge. It is well known that when the number of available samples $N$ is smaller than the dimension of the covariance matrix $n$, the estimated matrix is singular. In \cite{Fan2013}, a simulation is used to illustrate the challenge of estimating the covariance matrix when $N$ is much smaller (half) than $n$ and also when $N$ is an order of magnitude (10 times) larger. Various regularization techniques have been proposed to estimate covariance matrices in the literature, and we refer to \cite{fan2015overviewestimationlargecovariance} for an overview of some of these techniques. To our knowledge, other manifold based methods do not account for the inexactness of gradient information. Alternatively, one may compute sample covariances for a fixed sample size and solving the optimization problem using the sample covariances. In this setting, the error does not diminish and solving the optimization can be equated to solving a different problem altogether as the sample covariances are no longer representative of the true covariance matrices. Given the difficulty in estimating the covariance matrix, we work with RiGD to accommodate bias induced by the estimation process on Eq. \eqref{Jdiv} and Eq. \eqref{eq5} and incorporate the following regularization for covariance matrix stability 
\begin{align}\label{cov_est_reg}
    \tilde{K'} = (1 - \lambda) \tilde{K} + \lambda I, 
\end{align}
where $I_{n \times n}$ is the identity matrix, $\lambda \in (0,1)$ is a regularization parameter and $\tilde{K}$ as given by Eq. \eqref{cov_est}. The eigenvalues of a sample covariance matrix tend to be biased lower for smaller (true) eigenvalues and higher for larger (true) eigenvalues. Regularization shrinks the eigenvalues towards $1$ in Eq. \eqref{cov_est_reg}. Note that we choose shrinkage towards the identity and one may shrink towards the average eigenvalue of the sample covariance. A total of $30,000$ samples are generated for the optimization results presented in Fig.(4). In the sample covariance setting, the number of samples is increased by $200$ at every iteration thus at iteration $k$, $200 * k$ samples are used. For the uniform perturbation, the error diminishes naturally as the number of iterations is increased with rate $\mathcal{O}(1/k)$ (Eq.~\eqref{unif_bias}). Lastly, we also assume true covariance information as an ideal setting for comparison and use the Steepest Descent Algorithm from the ManOpt \cite{manoptManopt} software package to maximize Eq. \eqref{Jdiv}. The $2500 \times 2500$ covariance matrices (corresponding to $50 \times 50$ pixel images) $K_1, K_2$ are calculated with correlation length $\sigma_1 - \sigma_2 = 0.55 - 0.30 = 0.25$ pixels and no signal in the mean i.e. $s = 0$ in Eq. \eqref{Jdiv} (see \cite{kupinski_15} for expressions relating these variables to the covariance matrix). The initial point is chosen such that $T_0 \in \mathrm{Gr}(25,2500)$ and is generated by a QR-factorization of a random matrix of dimension $25 \times 2500$ for all algorithms. Theoretically, one expects the Frobenius norm of the gradient to converge to $0$ at a rate of $1/K$ for RiGD. For RiGD in Figure \ref{fig:rigd_plots}, a fixed step size of $\eta = 0.2$ is chosen under uniform perturbation and $\eta = 1.5$ the sample covariance setting. Parameter initializations for the RiGD-LS are $\sigma = 10^{-4}$, $\beta = 0.7$ and $\eta_0 = 2.0$ under the sample covariance setting ($\tilde{K'}$) and $\sigma = 10^{-4}$, $\beta = 0.7$ and $\eta_0 = 0.5$ under uniform perturbation ($\hat K$). Behavior of the relative error parameter $\delta$ as a function of iterates is presented in Figure \ref{fig:rigd_relerror} for the line search variants with varying choices of the shrinkage factor $\beta$.
\begin{figure*}[t]
    \centering
    \includegraphics[width=0.7\textwidth]{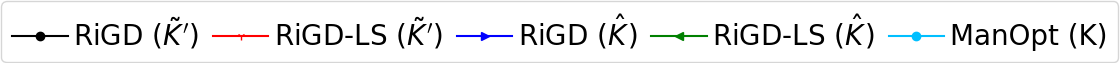}
    \subfloat[Merit function convergence]{\includegraphics[width=0.49\textwidth]{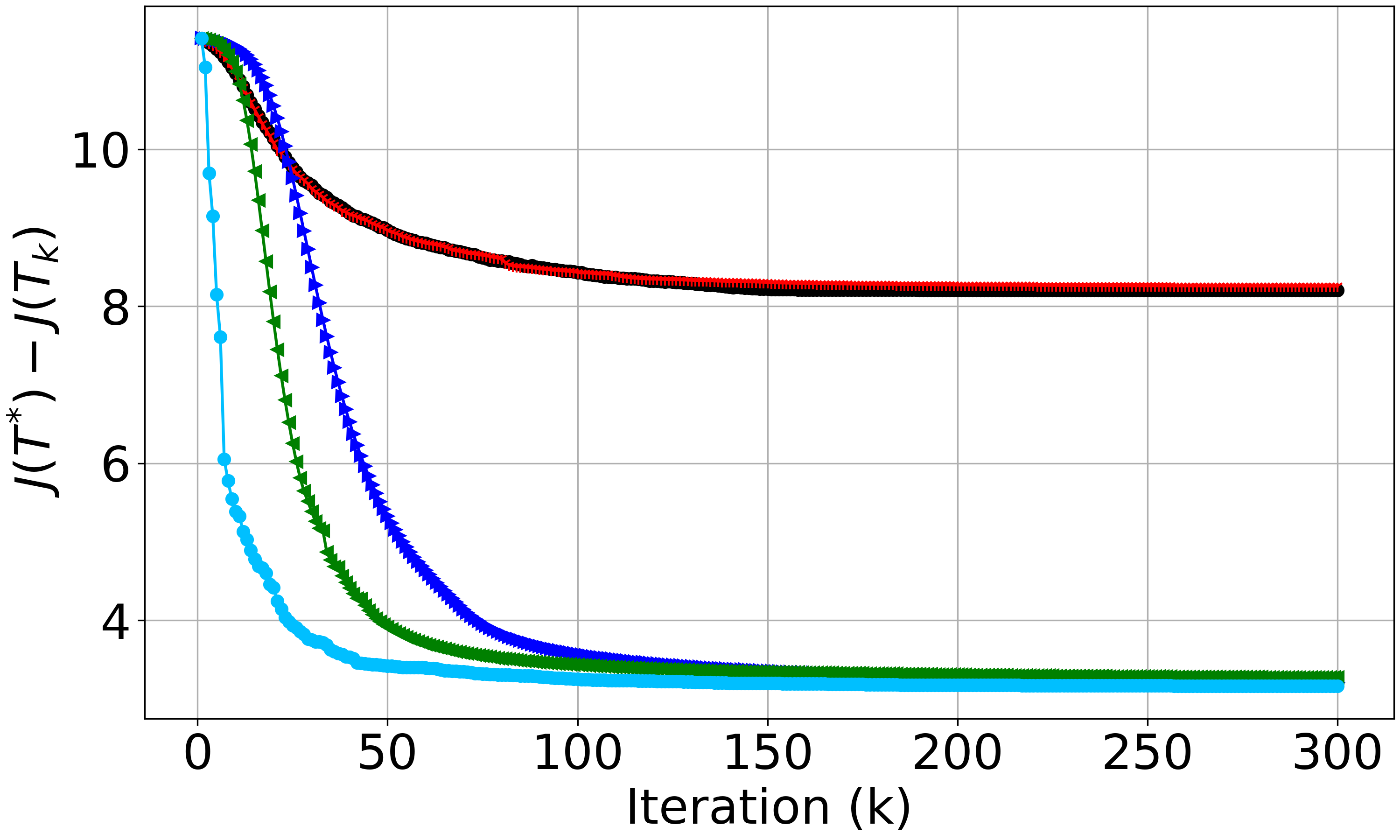}}
    \hspace{0.1mm}
    \subfloat[Gradient convergence]{\includegraphics[width=0.49\textwidth]{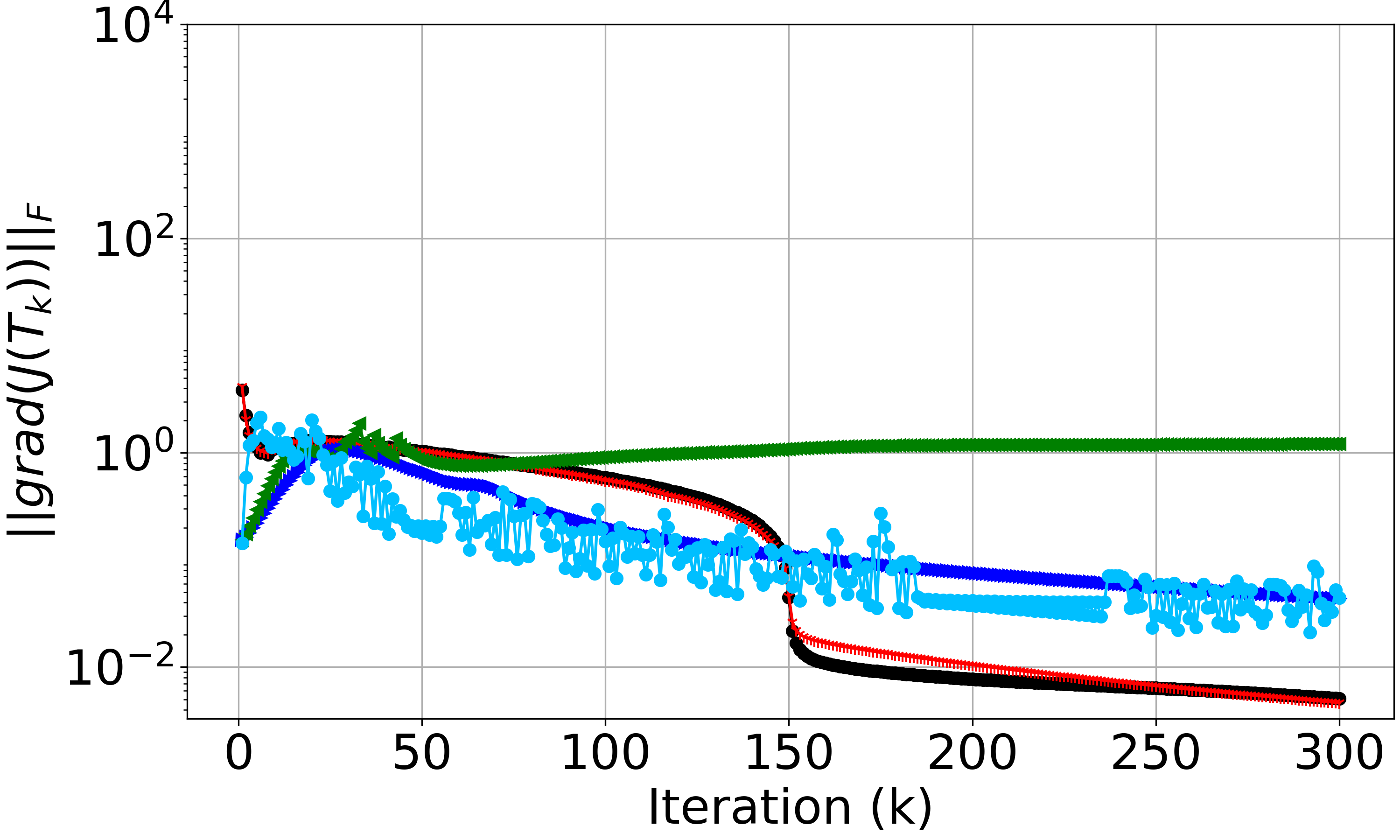}}
    
    \caption{Convergence analysis of Algorithms \ref{gigd}, \ref{gigd-ls} and ManOpt under varying perturbations to the covariance while solving \eqref{cqo_prob}. The theoretical optimal $T^*$ is the Fukunaga-Koontz which is used to compare iterative solutions in (a) where the sample covariance Eq. \eqref{cov_est_reg} with regularization ($\lambda = 0.6$) is black/red and blue/green is a uniform perturbation Eq. \eqref{unif_bias}. Here RiGD-LS is closer to optimal compared to RiGD for the uniform perturbation but not the sample covariance. The closest convergence to $T^*$ is ManOpt using the true covariance (teal). In (b) convergence of the Frobenius norm of the gradient Eq. \eqref{eq5} is presented.}
    \label{fig:rigd_plots}
\end{figure*}

Figure \ref{fig:rigd_plots} shows the convergence of the RiGD algorithm for four cases: fixed step size and linesearch with either uniformly perturbed ($\hat K$) or sample covariance matrices ($\tilde{K'}$). The teal curve in Figure \ref{fig:rigd_plots} assumes knowledge of the true covariance ($K$) although in practice, this is rarely the case and only sample estimates of covariance are available. When $s = 0$ in Eq. \eqref{Jdiv} and $K_1 \neq K_2$, it has been shown that the Fukunaga-Koontz transform is optimal and is obtained via an eigen-decomposition of $K_2^{-1}K_1$ (see \cite{fukunagakoontz} for details). In the more general setting when $s \neq 0$ and $K_1 \neq K_2$, an analytical optimal solution is not available. Figure \ref{fig:rigd_plots} (a) compares the convergence of each method in the objective function sense Eq. \eqref{Jdiv} to the Fukunaga-Koontz optimum $J(T^*)$. Under uniform perturbation, the line search variant converges faster to the optimum compared to fixed step-size and achieves a higher objective function value. However, under the sample covariance setting, convergence to the optimum is similar with the terminal objective value being marginally lower than the fixed step-size variant. When true covariance matrices are used ($K$), convergence is faster in the objective function sense. Figure \ref{fig:rigd_plots} (b) compares the convergence of the Frobenius norm of the gradient for each method. For the sample covariance setting ($\tilde K'$), the norm of the gradient diminishes at a rate of $1/K$ - our theoretical expectation. RiGD-LS converges faster as compared to the fixed step size variant and is advantageous from a user perspective since Lipschitz information is not required. Similar behavior can be observed under the uniform perturbation setting ($\hat K$).

\begin{figure}[h]
\centering
\begin{minipage}{0.48\textwidth}
\centering
\begin{tabular}{cc}
\hline
\textbf{Algorithm} & \textbf{Time (s)} \\
\hline
RiGD ($\tilde{K}'$)       & 1187.95 \\
RiGD ($\hat{K}$)          & 362.93 \\
RiGD-LS ($\tilde{K}'$)    & 1217.92 \\
RiGD-LS ($\hat{K}$)       & 456.47 \\
\hline
\end{tabular}
\captionof{table}{Computer time in seconds for the RiGD and RiGD-LS algorithms under both sample covariance and uniform perturbation settings. The computation was run on a machine with an Apple M1 Pro chip and 16.0 GB RAM.}\label{tab:computertime}
\end{minipage}
\hfill
\begin{minipage}{0.45\textwidth}
The computer time is reported for RiGD under the two different settings both with and without line search in Table \ref{tab:computertime}. In Table \ref{tab:auc_tab} the Area Under the Curve or AUC is reported for the final iterate of each algorithm and is computed using a two-alternative forced choice test on log-likelihood values (see \cite{kupinski_15} for details on the computation of the log-likelihood). An AUC of $100\%$ corresponds to perfect classification and zero detection error. When $s = 0$ in Eq. \eqref{Jdiv} and $K_1 \neq K_2$, Jeffrey's divergence and AUC are monotonically related i.e. an increase in Jeffrey's divergence corresponds to an increased ability of $T$ to discriminate between images from different classes.
\end{minipage}

\end{figure}

\begin{table}[h]
\captionsetup{width=\textwidth}
    \centering
    \caption{AUC computed for each algorithm at the final iterate on an independent sample of $5,000$ images from each class. $\tilde K'$ corresponds to sample covariance Eq. \eqref{cov_est}, $\hat K$ to uniformly perturbed covariance Eq. \eqref{unif_bias} and $K$ to true covariance matrices. FK denotes the AUC achieved by $T^*$ obtained via the Fukunaga-Koontz transform. }\label{tab:auc_tab}
\begin{tabular}{lcccccc}
\toprule
Algorithm & RiGD $(\tilde{K}')$ & RiGD-LS $(\tilde{K}')$ & RiGD ($\hat K$) & RiGD-LS ($\hat K$) \\
\midrule
AUC [\%] & 89 & 89 & 98 & 98 \\ 
\bottomrule
\end{tabular}

\vspace{0.5mm}  

\begin{tabular}{lcc}
\toprule
Algorithm & ManOpt $(K)$ & FK \\
\midrule
AUC [\%] & 97 & 99 \\
\bottomrule
\end{tabular}
\end{table}

\begin{figure}[H]
\centering

\begin{subfigure}{0.48\textwidth}
    \centering
    \includegraphics[width=\linewidth]{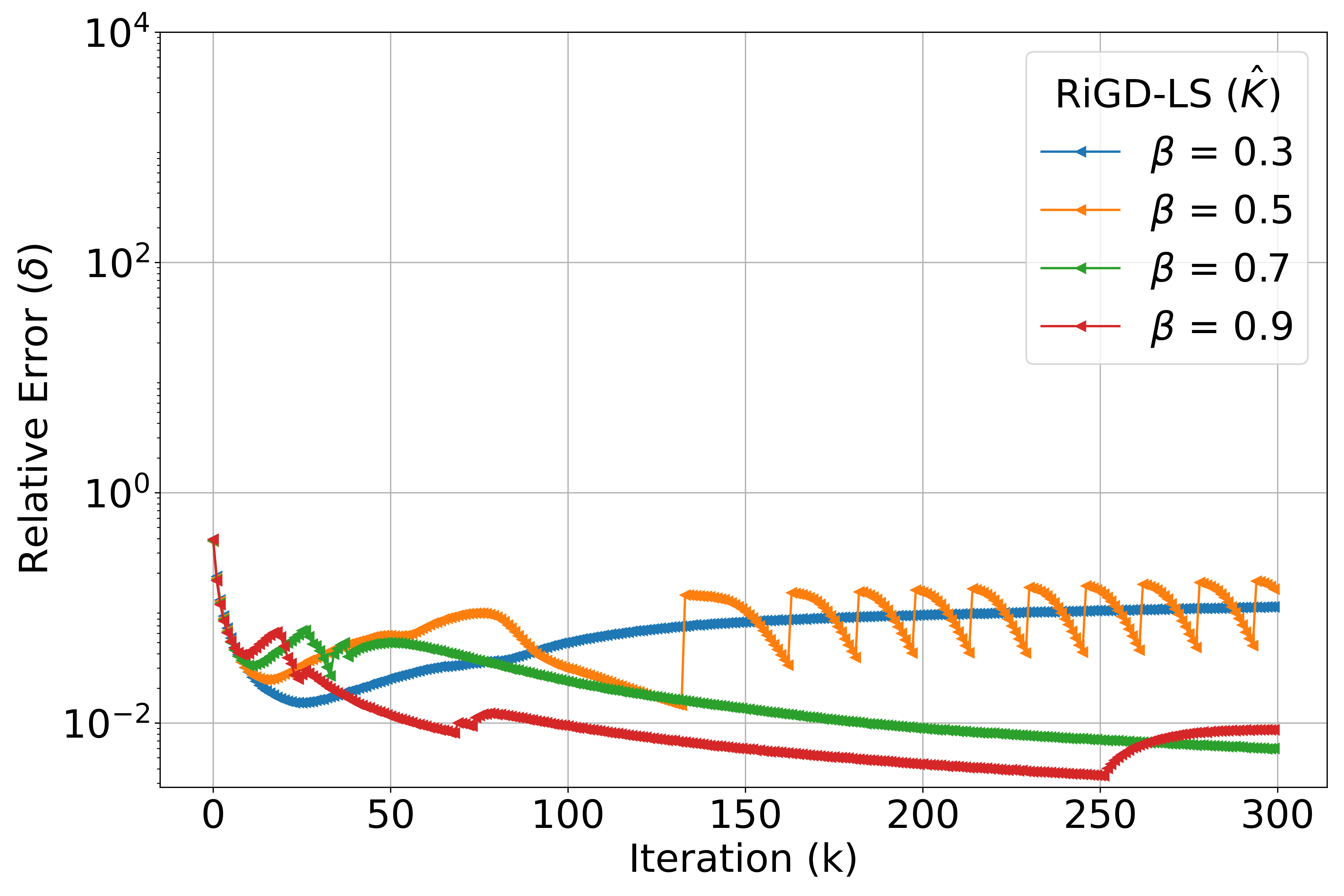}
    \caption{Uniform Perturbation}
\end{subfigure}
\hfill
\begin{subfigure}{0.48\textwidth}
    \centering
    \includegraphics[width=\linewidth]{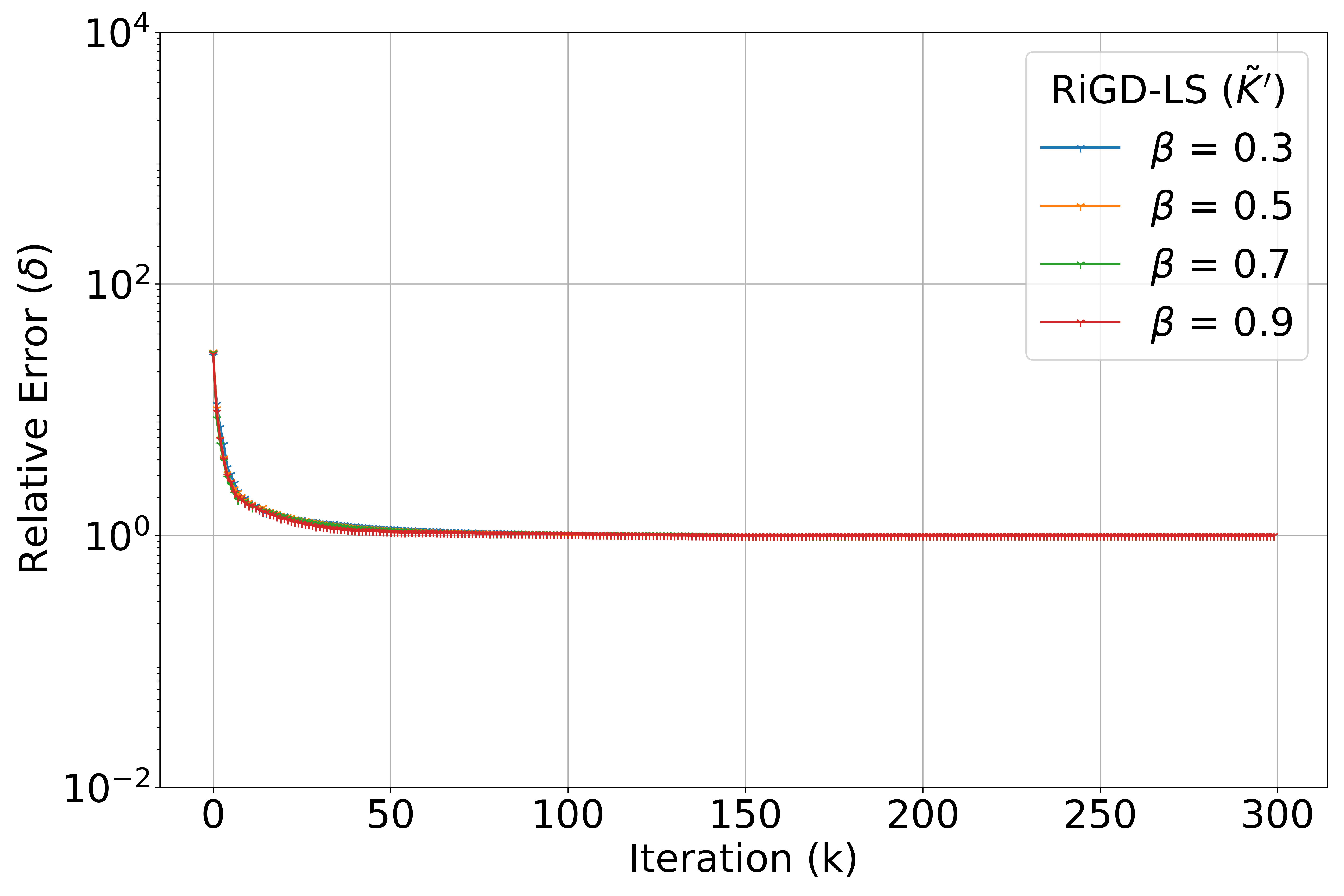}
    \caption{Sample Covariance}
\end{subfigure}
\caption{Relative error parameter $\delta$ as a function of iterates for varying choices of $\beta$, the line search parameter. In both the uniform perturbation and sample covariance settings, the parameter converges to a constant less than one.}
\label{fig:rigd_relerror}
\end{figure}

\begin{figure*}[t]
    \centering
    \fbox{
    \begin{subfigure}[b]{0.32\textwidth}
        \includegraphics[width=\linewidth]{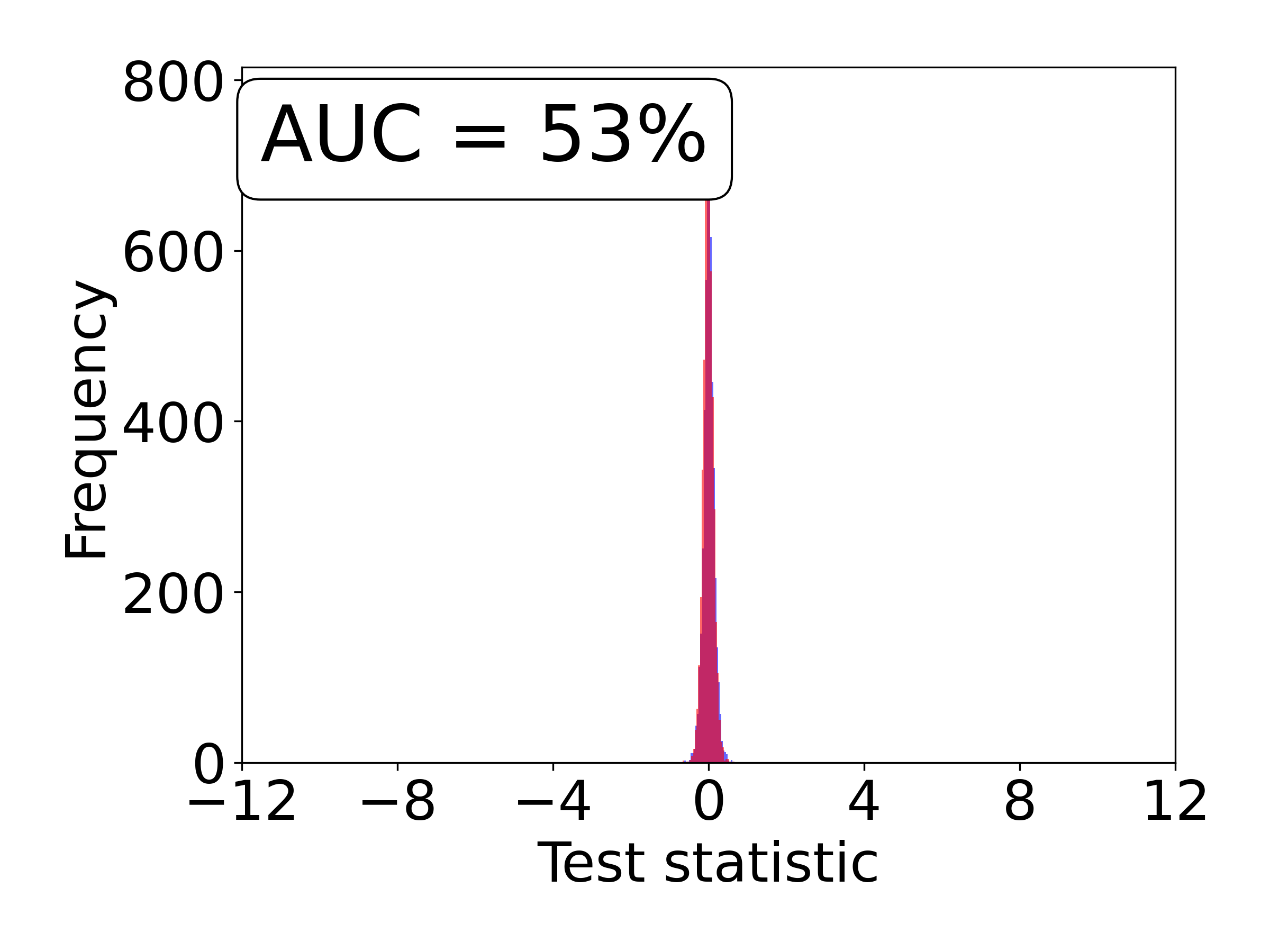}
        \caption{ManOpt: Iteration 1}
        \label{fig:exact_a}
    \end{subfigure}
    \begin{subfigure}[b]{0.32\textwidth}
        \includegraphics[width=\linewidth]{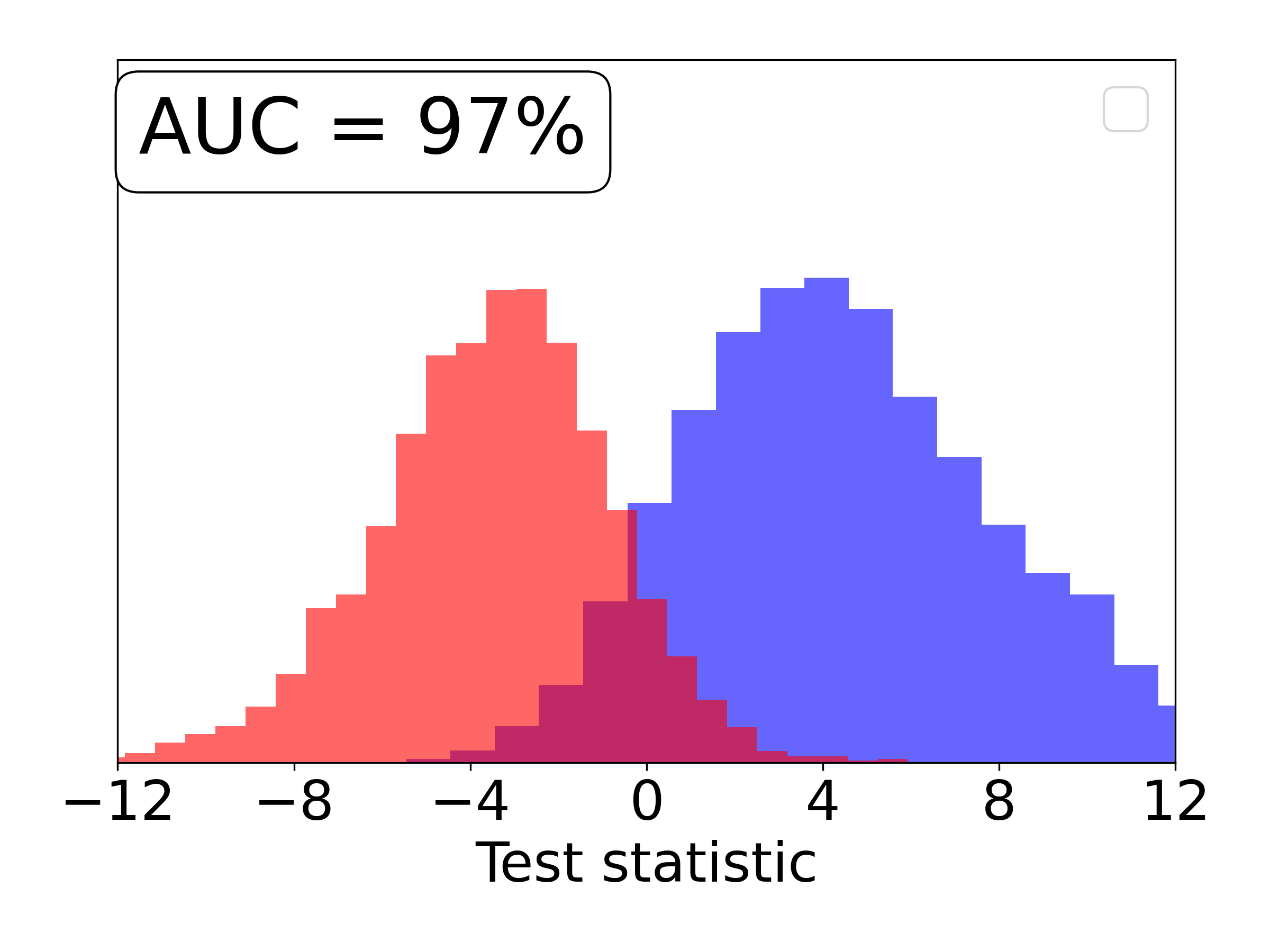}
        \caption{ManOpt: Iteration 50}
        \label{fig:exact_b}
    \end{subfigure}
    \begin{subfigure}[b]{0.32\textwidth}
        \includegraphics[width=\linewidth]{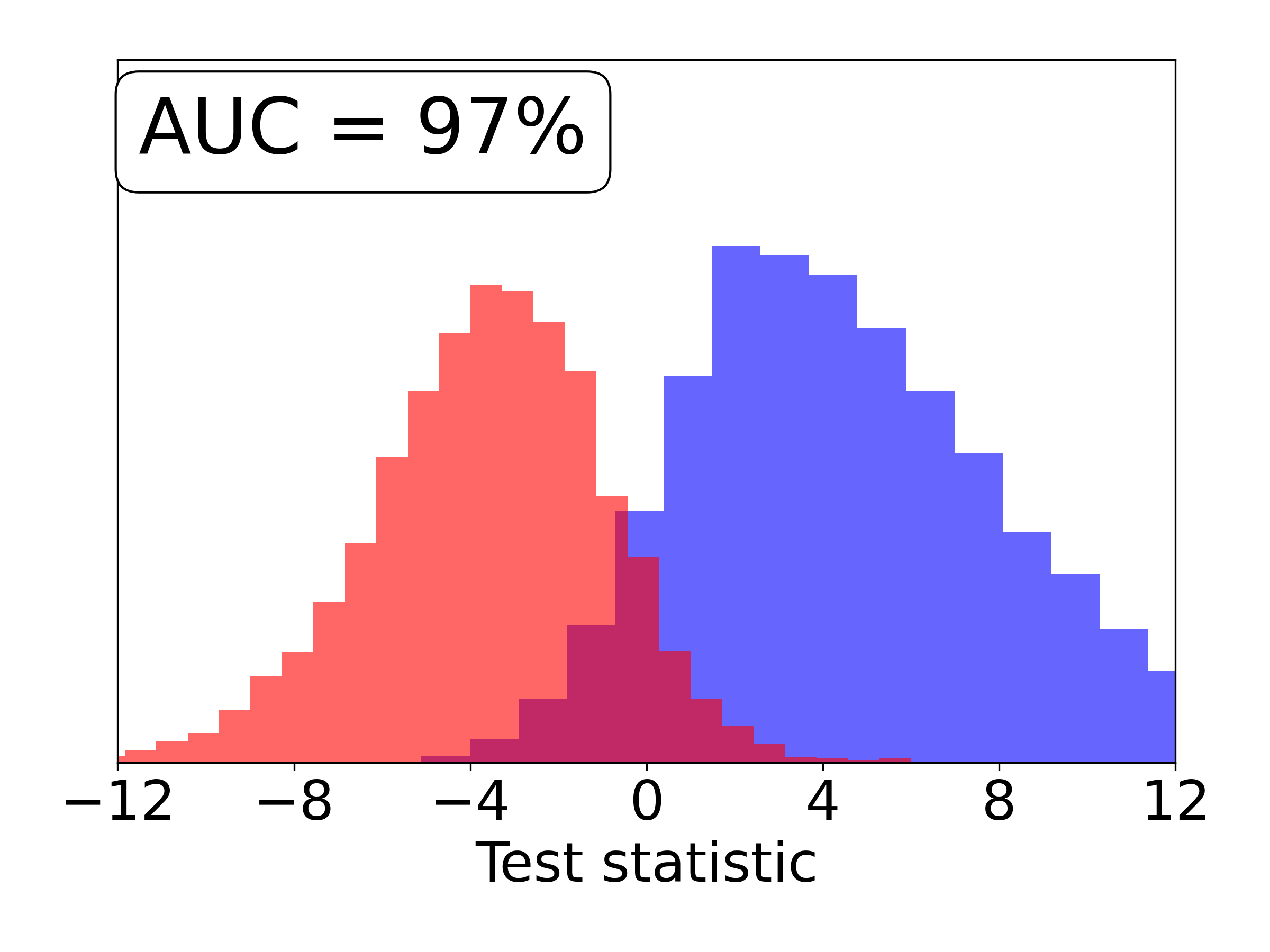}
        \caption{ManOpt: Iteration 300}
        \label{fig:exact_c}
         \end{subfigure}}\\
\fbox{\begin{subfigure}[b]{0.32\textwidth}
\includegraphics[width=\linewidth]{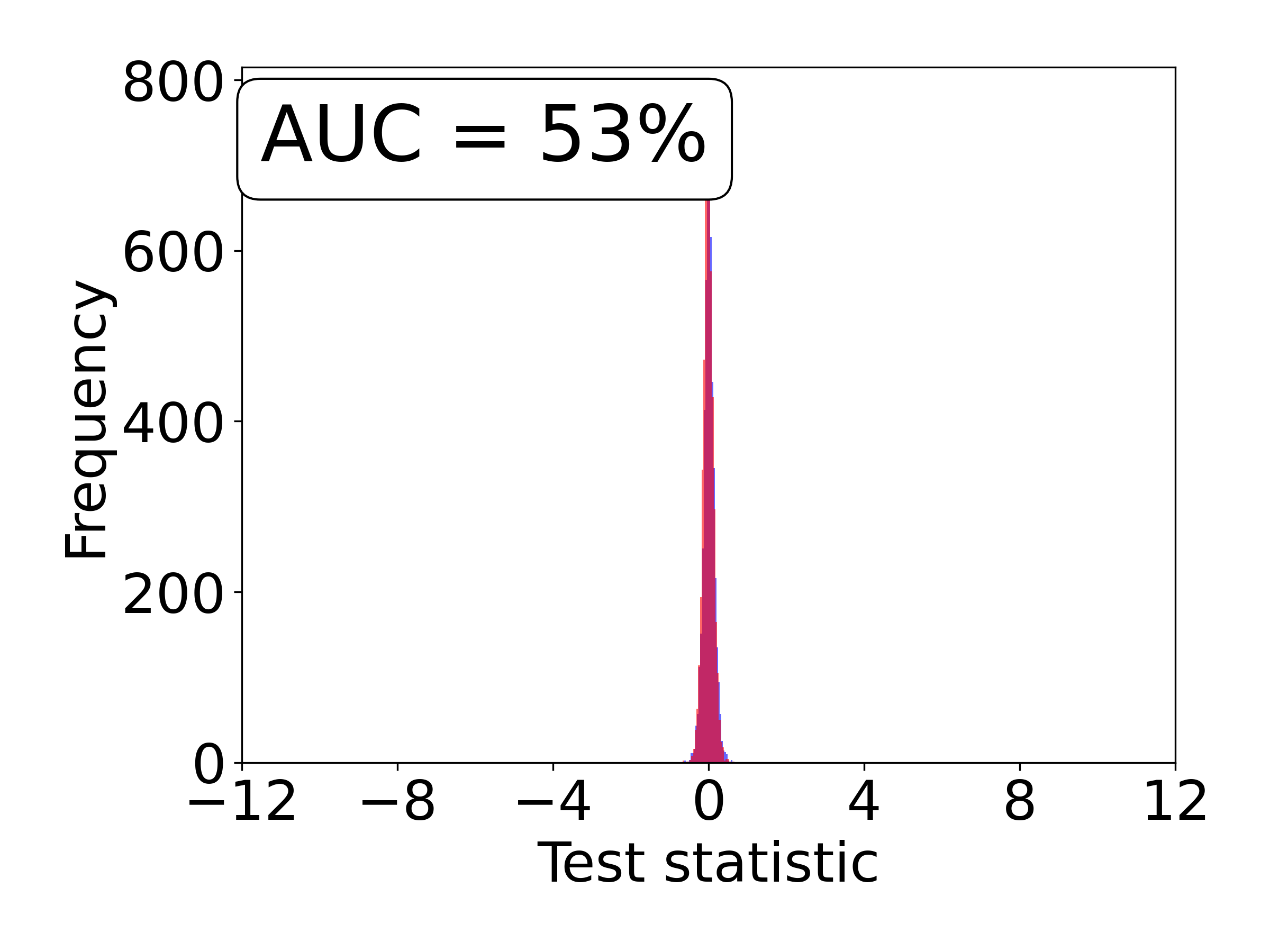}
        \caption{RiGD: Iteration 1}
        \label{fig:inexact_a}
    \end{subfigure}
\begin{subfigure}[b]{0.32\textwidth}
        \includegraphics[width=\linewidth]{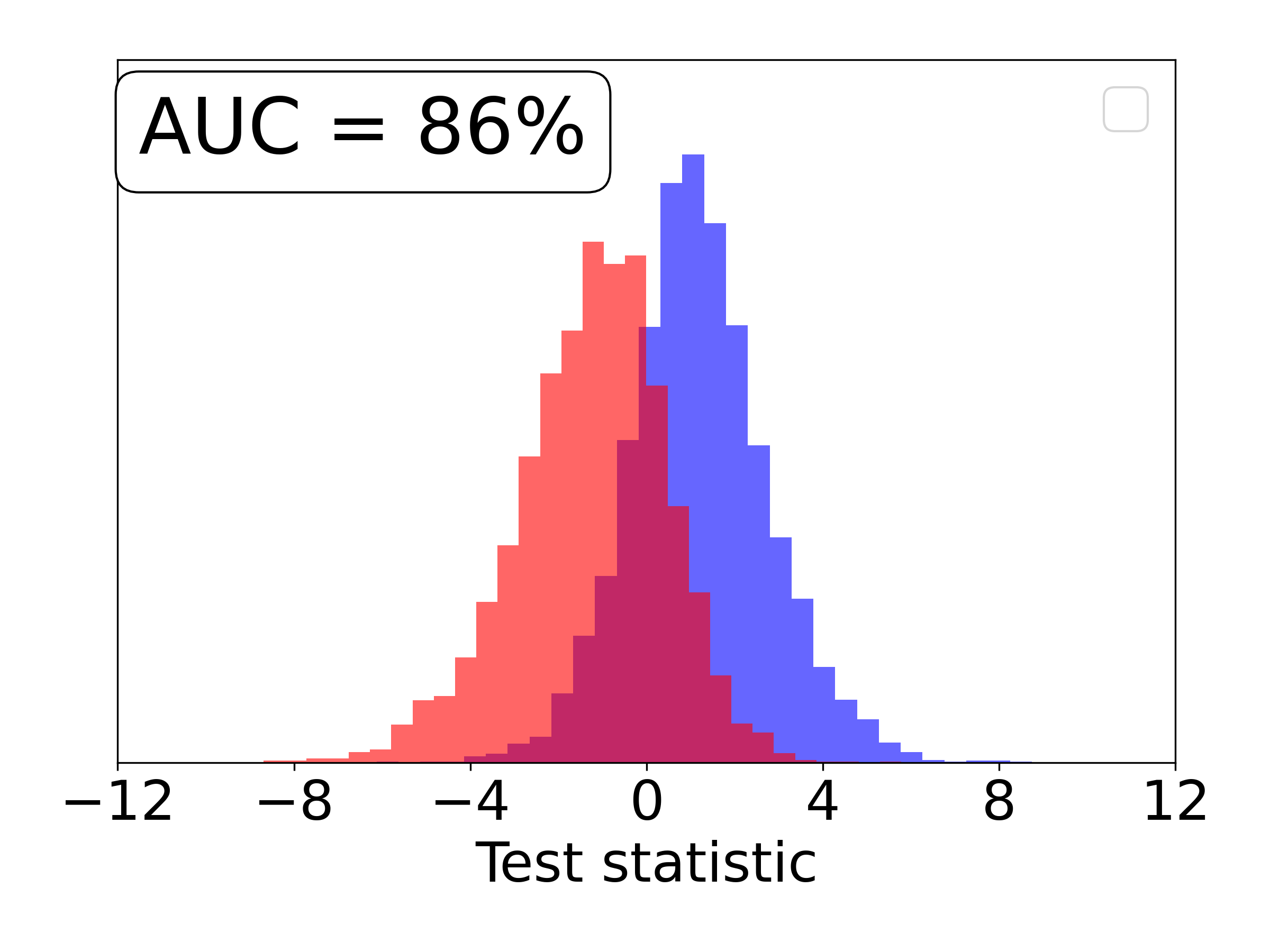}
        \caption{RiGD: Iteration 50}
        \label{fig:inexact_b}
    \end{subfigure}
\begin{subfigure}[b]{0.32\textwidth}
        \includegraphics[width=\linewidth]{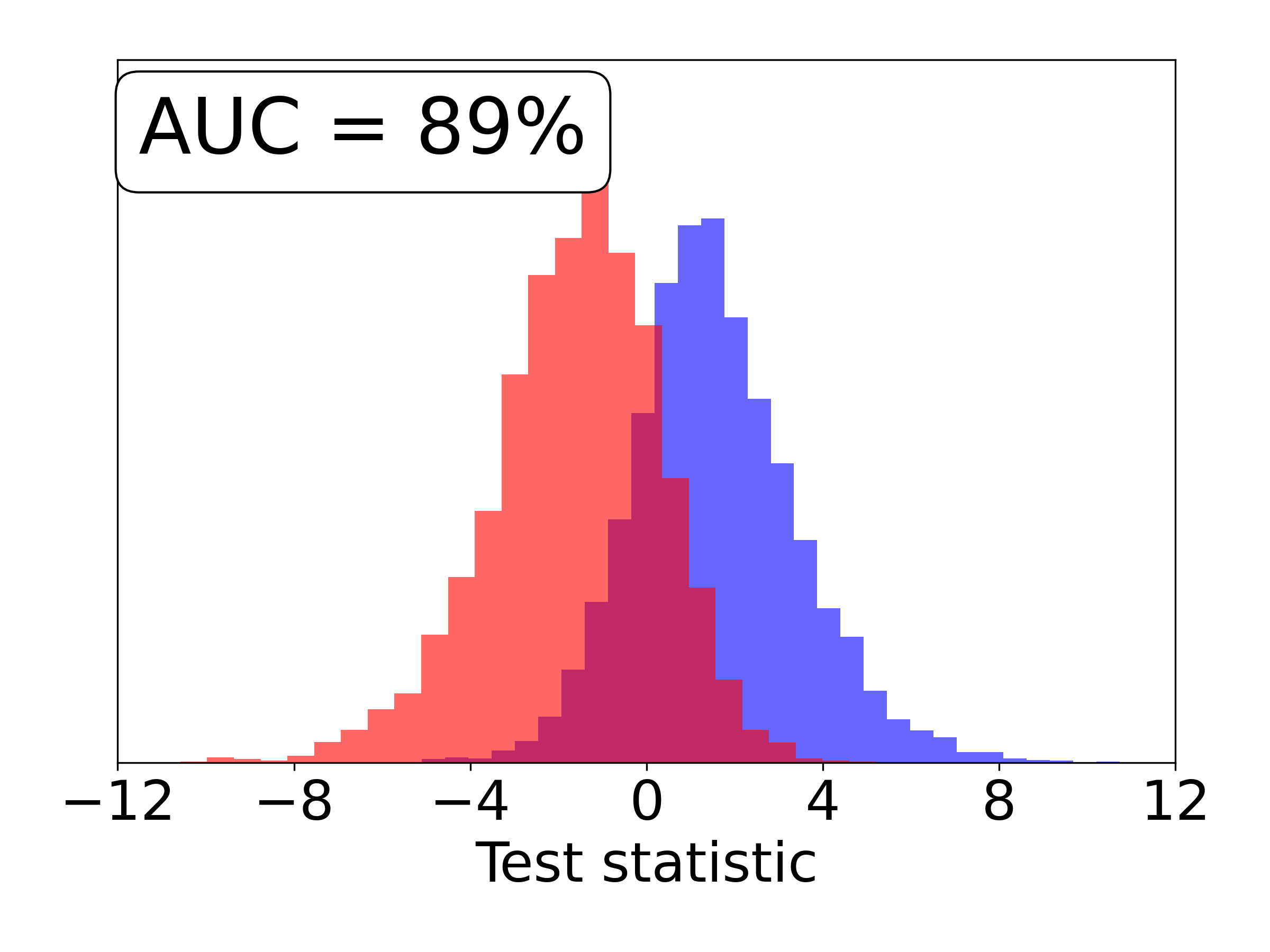}
        \caption{RiGD: Iteration 300}
        \label{fig:inexact_c}
    \end{subfigure}}
    \caption{The log-likelihood ratio from class 1 (blue) and class 2 (red) at three selected iterations. Results in (a-c, top row) are true covariances and ManOpt's Steepest Descent, and (d-f, bottom row) estimated covariances and RiGD.} 
    \label{fig:auc_plot}
\end{figure*}

For the sample covariance setting ($\tilde{K}'$) the line search variant of the RiGD achieves comparable convergence with similar AUC . On the other hand, under uniform perturbation RiGD-LS achieves the same AUC=$98\%$ for line-search and for fixed step size.  For RiGD-LS under uniform perturbation AUC=$98\%$ and for ManOpt AUC=$97\%$. This performance difference is an achievement for our algorithm considering that ManOpt assumes true covariance matrix information during optimization and RiGD works with estimated Eq. \eqref{cov_est_reg} and perturbed Eq. \eqref{unif_bias} covariance matrices.

Now, in Figure \ref{fig:auc_plot},  we visualize the detection performance of the matrix $T$ at select iterations. In Figures \ref{fig:exact_a} - \ref{fig:exact_c}, a histogram of log-likelihood values is plotted at successive iterates of the ManOpt algorithm assuming true covariance matrices are known. On the other hand, Figures \ref{fig:inexact_a} - \ref{fig:inexact_c} correspond to the iterates of RiGD using sample covariance matrices. Both algorithms are initialized at random matrices where AUC=$53\%$ computed using $5,000$ independently generated sample images. An improvement in the AUC is observed over successive iterations of both algorithms. As expected, AUC improves more rapidly when true covariance matrices are known and the ManOpt AUC is higher than RiGD at 50 and 300 iterations.

A key advantage of treating linear data reduction as an optimization problem is that the solution $T_k$ can offer insight into the correlation structure of the data. The ability to visualize this solution is key, as the rows of $T$ yield insight into the subspace that retains detection information. Thus far we have compared solutions from different algorithms according to the merit function value, gradient convergence, and AUC. Although, visualization of $T_k$ is qualitative it offers the opportunity to understand the resemblance between the correlation structure of the images and the optimized subspace solution. This visualization is complicated by the invariance condition in Eq. \eqref{inherent_grass}. Therefore, the first five eigenvectors of the ratio matrix $R_{21} = C_2^{-1}(T^*)C_1(T^*)$ are back-projected by the solutions for visual comparison. In Figure  \ref{fig:rows_viz_combined} each of the five $50 \times 50$ images correspond to a back-projected eigenvector. In practice, inspection of these solutions provides insight into the optimal linear combinations of the data for the detection task. In Figure \ref{fig:rows_viz_combined} (a) the eigenvectors are back-projected by the Fukunaga-Koontz transform $T^*$ and in Figure \ref{fig:rows_viz_combined} (b) by the final iterate of Algorithm \ref{gigd}. The image correlation structure can be visually recognized in the Fukunaga-Koontz transform which is the theoretic solution to maximize Jeffrey's divergence. Here, a combination of high and low frequency patterns at varying orientations are observed. A similar structure, although with longer correlation structure, can be seen in Figure \ref{fig:rows_viz_combined} (b) which utilizes the matrix $T_k$ computed using RiGD.  The optimal subspace learned via RiGD shows qualitative similarity to the correlation structure of the images in Figure \ref{fig:sample_images}. In practice, when the image's correlation structure is unknown the RiGD solution can be visualized to offer insight. 

\begin{figure}[h]
    \centering

    \fbox{\begin{subfigure}[t]{\textwidth}
        \centering
        \includegraphics[width=\textwidth]{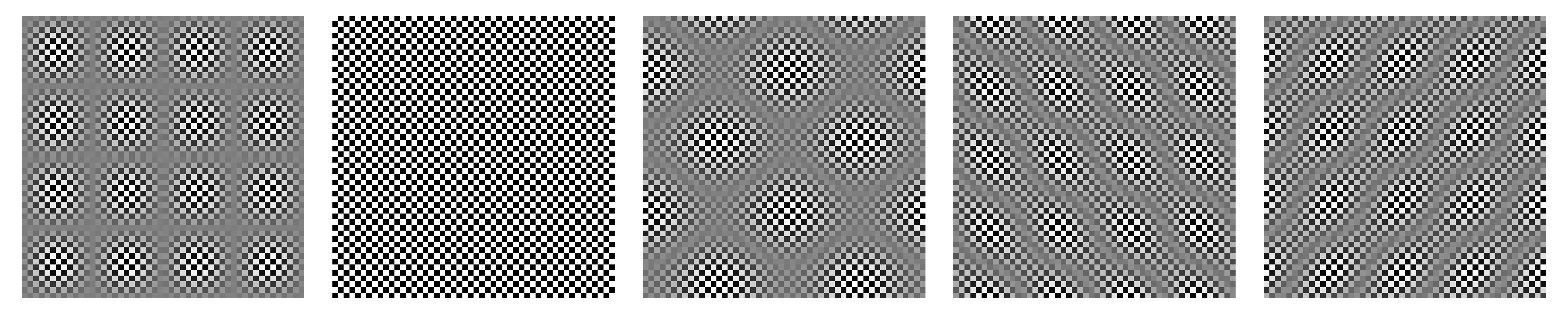}
        \caption{ $T^*$, the Fukunaga-Koontz transform.}
        \label{fig:rows_viz_teig}
    \end{subfigure}}

    \vspace{0.5cm}  

    \fbox{\begin{subfigure}[b]{\textwidth}
        \centering
        \includegraphics[width=\textwidth]{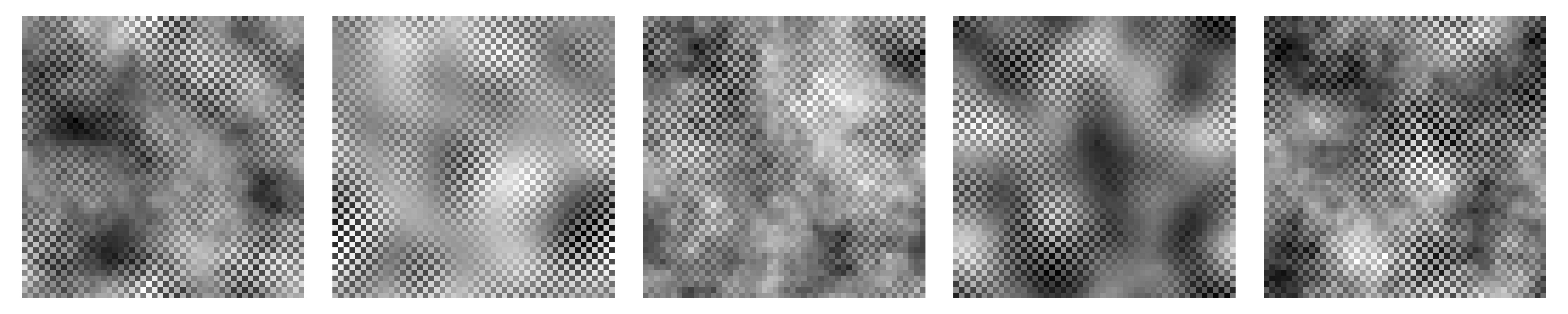}
        \caption{ RiGD $T_k$  (Algorithm \ref{gigd}) under uniform covariance perturbation (Eq. \eqref{unif_bias}).}
        \label{fig:rows_viz_rigd}
    \end{subfigure}}

    \caption{The first $5$ eigenvectors of the channelized ratio matrix $C_2^{-1}(T^*)C_1(T^*)$ back projected by (a) $T^*$ and (b) RiGD $T_k$.  This visualization of the Grassmannian solutions offers insight into the correlation structure of the images.}
    \label{fig:rows_viz_combined}
\end{figure}

\section{Conclusion}
In this work, we introduced the RiGD algorithm, a novel approach to manifold optimization that explicitly accounts for inexact gradient information. Our analysis demonstrates that RiGD achieves an O(1/K) convergence rate under standard assumptions—matching the convergence performance of methods that rely on exact gradients. We further developed a line search variant (RiGD-LS) that adapts step sizes without requiring knowledge of the Lipschitz constant, enhancing robustness and practical usability.

RiGD provides a principled and practical framework for manifold optimization in settings where gradient inexactness is inevitable, such as HDLSS imaging. Through numerical experiments of binary image classification with quadratic observers, we confirmed that RiGD remains effective even when gradients are biased due to covariance estimation or perturbations. Our approach achieved detection performance comparable to methods with access to exact statistics, while offering interpretable solutions that reflect the underlying correlation structure of the data.

\section*{Acknowledgments}

This work is supported by the Office of Naval Research (ONR) (N00014-24-1-2074) grant for Basic and Applied Scientific Research. Thank you to our colleague Eric Clarkson for his comments which improved this manuscript.

\section*{Declarations}
\noindent \small{{\bf Author contribution}\ \ The manuscript was initially written by UT and edited by MK and AJ who also provided research conditions and
guidance. The technical implementations,
such as coding algorithms and setting up reproducible environments were done by UT.
All authors read and approved the final manuscript.}\\

\noindent \small{{\bf Availability of data and materials}\ \  The basic code of this work is publicly available on OSF \hyperlink{OSF}{https://doi.org/10.17605/OSF.IO/P48AZ}.}

\nocite{*}
\bibliographystyle{plain}
\bibliography{main_revised}

\end{document}